\documentclass[english,12pt,oneside]{amsproc}
\usepackage[english]{babel}
\usepackage{a4wide}
\usepackage{amsthm}
\usepackage{graphics}
\usepackage{amsfonts, amssymb, amscd, amsmath}
\usepackage{latexsym}
\usepackage[matrix,arrow,curve]{xy}
\usepackage{mathabx}%,mathtools}
\usepackage{color}
\usepackage{pbox}
\usepackage{tikz}
\usetikzlibrary{matrix,decorations.pathreplacing,positioning}
\usepackage{hyperref}

\DeclareMathOperator{\Cone}{Cone} \DeclareMathOperator{\pt}{pt}

 \DeclareMathOperator{\lk}{lk} \DeclareMathOperator{\st}{st}

\DeclareMathOperator{\Hilb}{Hilb}
\DeclareMathOperator{\conv}{conv} \DeclareMathOperator{\Ob}{Ob}
\DeclareMathOperator*{\colim}{colim} 
\DeclareMathOperator{\cl}{cl} \DeclareMathOperator{\scl}{sc}
\DeclareMathOperator{\op}{op} \DeclareMathOperator{\inc}{inc}
\DeclareMathOperator{\Quil}{Quil} \DeclareMathOperator{\Nerve}{Nerve}
\DeclareMathOperator{\pr}{pr} \DeclareMathOperator{\SubObj}{SubObj}
 \DeclareMathOperator{\neur}{neur}
\DeclareMathOperator{\cone}{Cone} \DeclareMathOperator{\core}{core}

\newcommand{\ko}{\Bbbk}
\newcommand{\Fo}{\mathbb{F}}

\newcommand{\Ro}{\mathbb{R}}

\newcommand{\eqd}{\stackrel{\text{\tiny def}}{=}}

\newcommand{\ca}[1]{\mathcal{#1}}

\newcommand{\dd}{\partial}
\newcommand{\Ca}{\mathcal{C}}
\newcommand{\I}{\mathcal{I}}
\newcommand{\F}{\mathcal{F}}

\newcommand{\Rn}{R_{\neur}}
\newcommand{\In}{I_{\neur}}

\newcommand{\ST}[1]{\mbox{\upshape\small #1}}
\newcommand{\cat}{\ST{CAT}}

\newcounter{stmcounter}[section]
\newcounter{thcounter}

\numberwithin{equation}{section}

\theoremstyle{plain}
\newtheorem{cor}[stmcounter]{Corollary}

\newtheorem{thm}[thcounter]{Theorem}

\newtheorem{prop}[stmcounter]{Proposition}
\newtheorem{lem}[stmcounter]{Lemma}

\theoremstyle{definition}
\newtheorem{defin}[stmcounter]{Definition}

\theoremstyle{remark}
\newtheorem{ex}[stmcounter]{Example}
\newtheorem{rem}[stmcounter]{Remark}
\newtheorem{con}[stmcounter]{Construction}

\begin{document}

\title{Topology of nerves and formal concepts}

\author{Anton Ayzenberg}
\address{Laboratory of applied geometry and topology, Faculty of computer science, Higher School of Economics}
\email{ayzenberga@gmail.com}

\date{\today}
%\thanks{This work was supported by the Russian Science Foundation under grant 18-71-00009.}



\begin{abstract}
The general goal of this paper is to gather and review several methods from homotopy and combinatorial topology and formal concepts analysis (FCA) and analyze their connections. FCA appears naturally in the problem of combinatorial simplification of simplicial complexes and allows to see a certain duality on a class of simplicial complexes. This duality generalizes Poincare duality on cell subdivisions of manifolds. On the other hand, with the notion of a topological formal context, we review the classical proofs of two basic theorems of homotopy topology: Alexandrov Nerve theorem and Quillen--McCord theorem, which are both important in the applications. A brief overview of the applications of the Nerve theorem in brain studies is given. The focus is made on the task of the external stimuli space reconstruction from the activity of place cells. We propose to use the combination of FCA and topology in the analysis of neural codes. The lattice of formal concepts of a neural code is homotopy equivalent to the nerve complex, but, moreover, it allows to analyse certain implication relations between collections of neural cells.

%
%The particular goals are as follows. (1) A simplification algorithm for a simplicial complex, which, theoretically, can be used to simplify homology computations is introduced. The algorithm is based on Quillen's fiber theorem and is related to other modern algorithms directed to solving similar tasks. (2) The duality of nerve complexes is introduced. It generalizes the polar duality of polytopes and Poincare duality of cell subdivisions of closed manifolds in a natural way. (3) We review the basics of FCA, and show that duality of nerve complexes is a trivial consequence of the duality between objects and attributes in this theory. It seems that many methods of the modern topological data analysis have already been developed in FCA. (4) Topological methods, in particular, the nerve theorem, are widely applied in brain studies, in particular in the problem of reconstruction of the stimuli space from the activity of place cells. We give a brief review of these methods and make a suggestion that FCA can be equally important in these studies. (5) Motivated by problems in brain study, we introduce the notion of an enriched formal context. We review the classical proofs of Nerve theorem and Quillen's theorem in the rather elementary language of topological formal contexts.
\end{abstract}

\maketitle

\section{Introduction}\label{secIntro}
\subsection{Brief overview}
The general goal of this paper is to gather and review several basic methods from homotopy topology and formal concepts analysis (FCA). It appears that methods which are recently developing in topological data analysis (TDA) can be treated in terms of the theory of formal concepts: this allows to apply a well-developed algorithmic machinery of FCA in the area of computational topology. It may also be the case that the topological insight can open new directions in FCA.

Both homotopy topology (e.g. \cite{CurIts,CurItsNeuralRing,Dabag,DabagNoCov,Manin}) and FCA (e.g. \cite{EFP,EAGN}) had found applications in brain studies. These approaches to the analysis of neural codes are compatible in the sense which is described in this paper. We propose the notion of a topological formal context, which allows to quantify relations between a topological space of stimuli and a topological space of neurons. The notion of a topological formal context is applied to revisit the classical proofs of two important theorems in homotopy topology: the Nerve theorem of P.\,S.\,Alexandrov and Quillen--McCord theorem. Both theorems have similar proofs and both can be used in applications, in particular in the analysis of neural codes.

\subsection{Motivation and details}
The paper was originally motivated by the following problem. In computational topology, there is often a necessity to estimate a shape of some topological space numerically. For this purpose, the topological space is covered by a finite collection of subsets, then the simplicial homology modules of the nerve of this covering are computed. Other simplicial complexes can be used instead of the nerve, for example the Vietoris-Rips complex. However, the nerves of covers and Vietoris--Rips complexes may have a large dimension compared to the dimension of the given space (see Example~\ref{exCoverSimplify}). The dimensionality makes the direct computation of simplicial homology time consuming. Therefore, some preliminary methods are applied to simplify the simplicial complex.

One of the solutions of the dimension problem is to use alpha complexes~\cite{AlphaEdel,AlphaPap} instead of nerves and Vietoris--Rips complexes. The alpha complex is a certain combination of the nerve of the covering with the Delaunay triangulation. In particular, in order to construct the Delaunay triangulation, a metric should be specified on a given data. In this paper we develop the approach based on a combinatorial simplification of a simplicial complex, which is more in the spirit of the work~\cite{BPP}. This approach does not require a metric: it can be applied to any combinatorial simplicial complex.

The idea is the following: we look at a simplicial complex $K$ as a partially ordered set, and remove all its simplices which do not affect the homotopy type of the poset. We end up with a subposet $\widetilde{K}$ of $K$ (see Construction~\ref{conMainConstruction}) which has the same homotopy type as~$K$. This idea is not new: in the homotopy theory of finite topological spaces, it is realized under the name of the core of a space~\cite{BarmakBook,May,Osaki,Stong}, also see Definition~\ref{defUpDownCore}.

In this paper we want to put accent on the fact, that the construction of $\widetilde{K}$ out of $K$ is a classical procedure studied in the formal concept analysis (FCA) and data mining. In the language of FCA (see Section~\ref{secFCA}), the poset $\widetilde{K}$ of $K$ is exactly the lattice of formal concepts of the formal context $(V,W,\ca{I})$, where $V$ is the vertex set, $W$ is the set of maximal simplices of $K$, and $\ca{I}$ is the relation of inclusion. There are several programs supporting calculations with formal contexts, and there exist algorithms for the generation of a formal concept lattice of a given context. We suppose that these algorithms may find application in topological data analysis.

The brain study is a natural choice of research area, where the combination of topology with FCA can be used. The applications of the Nerve theorem in the analysis of place cells' activity in the hyppocampus are well known~\cite{CurIts,Dabag,Manin}. Formal concepts analysis is applied to find hierarchical structures in fMRI data~\cite{EFP,EAGN}. We suppose that using topological FCA in the analysis of place cells' firing patterns has certain advantages compared to a straightforward use of the Nerve theorem. The concept lattice allows to see implications of the form $U_{i_1}\cap\cdots\cap U_{i_s}\to U_j$ (i.e. to answer the question whether $j$-th neuron fires whenever all the neurons $i_1,\ldots,i_s$ fire), which the nerve of the cover does not allow. On the other hand, the concept lattice still allows to reconstruct the homotopy type of the stimuli space according to Quillen--McCord theorem.

The paper has the following structure. The basic definitions related to topology and combinatorics of simplicial complexes and partially ordered sets are gathered in Section~\ref{secPrelim}. In Section~\ref{secSimplification}, we define two combinatorial simplification methods for simplicial complexes. The first one, the weeding $\widetilde{K}$ of $K$, is based on taking all possible intersections of maximal simplices in a simplicial complex. The second is based on consecutive removal of the nodes of the Hasse diagram if they have either in-degree $1$ or out-degree $1$. We call this operation Stong reduction. It is shown that the procedures are related to each other: the weeding can be obtained by Stong reduction.

In Section~\ref{secNervesDuality}, we review the construction of nerve complexes of convex polytopes which was introduced in~\cite{AB} for the purposes of toric topology. The notion of a nerve-complex is generalized to the class of sufficiently well-behaved cell complexes. There is a certain generalization of the Poincare duality: if $K$ is a simplicial complex and $L$ is the nerve of the cover of $K$ by its maximal simplices, then the weeding poset of $L$ is isomorphic to the weeding poset of $K$ with the order reversed (see Theorem~\ref{thmDualityOfNerves}). The proof of this statement is based on certain Galois correspondence between boolean lattices. This makes a bridge from combinatorial topology to the theory of formal concepts: Galois correspondences of this type constitute the theoretical basis of FCA. Some basic notions of FCA are gathered in Section~\ref{secFCA}. Any simplicial complex $K$ gives rise to a formal context, in which objects are vertices of $K$, attributes are maximal simplices of $K$, and the context is given by inclusion. The lattice of formal concepts of this context is isomorphic to the weeding $\widetilde{K}$. On the contrary, any formal context $\Ca$ gives rise to two simplicial complexes, which are both homotopy equivalent to the lattice of formal contexts, see Proposition~\ref{propContextComplexesHomot}. Further, the formal context of a covering is defined; it is shown that the lattice of formal concepts of this context contains more combinatorial information than the nerve of the covering.

In Section~\ref{secMouse}, we review the existing mathematical constructions related to the analysis of place cell activity in the hyppocampus of an animal, in particular the construction based on the Nerve theorem. It seems that these methods can benefit from using not only topology, but topology combined with FCA. The lattice of formal concepts of a covering contains not only the homotopical information about the stimuli space, but also allows to implement data mining techniques for analysing neural codes.

Section~\ref{secTopologicalFCA} is devoted to enriched formal contexts, the most important being topological and ordered formal contexts. The classical proofs of the Nerve theorem and Quillen--McCord theorem are formulated in this section in terms of topological formal contexts. Formally, these proofs do not rely on any theory developed in this paper, so the interested reader can jump to subsection~\ref{subsecProofs} with no harm.

This work does not contain essential new results in either topology, formal concept analysis or brain study. However we consider it important to bring the methods used in these areas together and give a self-contained review that underlines the relations between these areas. We try to keep the exposition simple to make it accessible for a broad range of readers, sometimes by sacrificing the level of generality and abstractness. In particular, we speak very little about simplicial sets (although the categorical approach is common in homotopy theory) or finite Alexandrov topologies (although the homotopy theory of finite spaces gives a more conceptual way of thinking about lattices). These notions are standard in this research area, however we prefer to keep things on geometrical and discrete-mathematical ground.% Hence we work with cell complexes, simplicial complexes, partially ordered sets, and their geometrical realizations.

\section{Preliminaries}\label{secPrelim}

\subsection{Simplicial complexes}
\begin{defin}
Let $V=[m]=\{1,2,\ldots,m\}$ be a finite set. A collection $K$ of subsets of $V$ is called a \emph{simplicial complex} if the following two conditions hold (1) $\varnothing\in K$; (2) if $I\in K$ and $J\subset I$ then $J\in K$. The third condition is also assumed to hold: for every $i\in V$, the singleton $\{i\}$ lies in $K$. The elements of $V$ are called vertices, and the elements of $K$ are called simplices of $K$. The number $\dim I\eqd |I|-1$ is called the dimension\footnote{Formally we have $\dim\varnothing=-1$ which may seem unnatural.} of a simplex $I\in K$.

For a simplicial complex $K$ on the set $[m]$ consider the topological space
\[
|K|=\bigcup_{I\in K, I\neq\varnothing}\triangle_I\subset \Ro^m,\quad \triangle_I=\conv\{e_i\mid i\in I\},
\]
where $e_1,\ldots,e_m$ is the standard basis of $\Ro^m$. The topological space $|K|$ is called the geometrical realization of $K$.
\end{defin}

\begin{con}
The basic constructions on simplicial complexes are as follows. If $K_1$, $K_2$ are simplicial complexes on the sets $[m_1]$ and $[m_2]$ respectively, then their join is defined by
\[
K_1\ast K_2=\{I_1\sqcup I_2\in [m_1]\sqcup[m_2]\mid I_1\in K_1,I_2\in K_2\}.
\]
This means that $K_1\ast K_2$ consists of all possible concatenations of simplices from $K_1$ and $K_2$. This construction is consistent with the topological join: $|K_1\ast K_2|\cong |K_1|\ast|K_2|$. The join $K\ast\pt$ of $K$ with the one-point simplicial complex is called the cone over $K$ and denoted $\cone K$.

The link of a simplex $I\in K$ is the simplicial complex
\[
\lk_KI=\{J\subset[m]\mid I\cap J=\varnothing, I\sqcup J\in K\}.
\]
The star of $I$ is the complex
\[
\st_KI=\{J\subset[m]\mid I\cup J\in K\}.
\]
The vertex sets in the definitions of link and star are chosen in a way that these complexes do not have ghost vertices. There holds $\st_KI=\lk_KI\ast \Delta_I$, where $\ast$ is the join of simplicial complexes, and $\Delta_I$ is the full simplex on the vertex set $I$. In particular, if $I$ is nonempty, the space $|\Delta_I|$ is contractible, hence $|\st_KI|=|\lk_KI|\ast|\Delta_I|$ is contractible as well.

If $i$ is a vertex of $K$, let $K\setminus i$ denotes the subcomplex formed by all simplices of $K$ which do not contain $i$. Then $K$ is obtained from $K\setminus i$ by attaching the cone $\st_Ki=\cone\lk_Ki$ to $\lk_Ki$:
\begin{equation}\label{eqStarSubcomplDecompose}
K=(K\setminus i)\cup_{\lk_Ki}\st_Ki.
\end{equation}
\end{con}

The study of general topological spaces can be reduced to the study of simplicial complexes in several ways. The classical way, which is relevant to our study, is based on the nerves of coverings.

\begin{con}
Let a topological space $X$ be covered by a collection of subsets $\ca{U}=\{U_1,\ldots,U_m\}$, that is $X=\bigcup_{i\in[m]}U_i$. The simplicial complex
\[
K_{\ca{U}}\eqd\{\{i_1,\ldots,i_k\}\subseteq[m]\mid U_{i_1}\cap\cdots\cap U_{i_k}\neq\varnothing\}
\]
is called \emph{the nerve of the covering} $\ca{U}$.

If the intersection $\bigcap_{i\in I}U_i$ is contractible for any $I\in K_{\ca{U}}$, $I\neq\varnothing$ (i.e. whenever the intersection is non-empty), the covering is called \emph{contractible}.
\end{con}

\begin{thm}[the Nerve theorem of P.\,S.\,Alexandrov~\cite{AlexNerve}]\label{thmNerveThm}
Let $\ca{U}$ be the covering of $X$ and $K_{\ca{U}}$ be its nerve. Assume that either all sets $U_i$ are open subsets of a paracompact space, or $X$ is a cell complex in which all possible intersections $U_{i_1}\cap\cdots\cap U_{i_k}$ are cell subcomplexes. If the covering $\ca{U}$ is contractible, then $|K_{\ca{U}}|$ is homotopy equivalent to $X$.
\end{thm}

The proof is given in subsection~\ref{subsecProofs}.

\subsection{Posets}

\begin{con}
Let $(S,\leqslant)$ be a partially ordered set (a poset). For two elements $s_1,s_2\in S$ we write $s_1<s_2$ if $s_1\leqslant s_2$ and $s_1\neq s_2$. Sometimes we write $s_1\leqslant_Ss_2$ to underline the poset in which the order relation is taken. All posets are assumed finite in this paper. If $S$ is a poset, let $S^{\op}$ denote the poset obtained from $S$ by reversing the order.

For an element $s\in S$ consider the posets
\[
S_{\leqslant s}=\{t\in S\mid t\leqslant s\},\qquad S_{<s}=\{t\in S\mid t<s\}.
\]
The posets $S_{\geqslant s}$ and $S_{>s}$ are defined similarly. The partial order on these subsets is induced from $S$.

With each poset $S$, one can associate a small category $\cat(S)$, whose objects are the elements of $S$ and there is exactly one morphism from $s_1$ to $s_2$ whenever $s_1\leqslant s_2$ and no morphisms otherwise. It is natural to denote this morphism simply by $(s_1\leqslant s_2)$. Therefore, the composition of morphisms is naturally defined by $(s_2\leqslant s_3)\circ(s_1\leqslant s_2)=(s_1\leqslant s_3)$, and the identity morphism of $s\in S=\Ob\cat(S)$ is $(s\leqslant s)$.
\end{con}

\begin{ex}
One can look at the simplicial complex as a particular example of a poset, since simplices are ordered by inclusion. However, it is natural to exclude the empty set from the consideration. So, for a simplicial complex $K$, we consider the poset $S_K=(K\setminus\varnothing,\subseteq)$.
\end{ex}

\begin{con}[Hasse diagram]
It is convenient to represent posets by their Hasse diagrams. The Hasse diagram $\Gamma_S$ of a poset $S$ is the directed graph on the vertex set $S$, which has a directed edge from $s_1$ to $s_2$ if $s_1<s_2$ and there is no element $s\in S$ with the property $s_1<s<s_2$. The (finite) poset can be restored from it Hasse diagram: we have $s_1\leqslant s_2$ if there is a directed pass from $s_1$ to $s_2$ in $\Gamma_S$. Similarly, any directed graph without directed cycles determines a poset.
\end{con}

\subsection{Geometrical realizations}

\begin{defin}\label{defPosetRealiz}
For a finite poset $S$, consider the simplicial complex $K(S)$ with the vertex set $S$ and simplices of the form $\{s_1,\ldots,s_k\}$ where $s_1<s_2<\cdots<s_k$ in $S$. This means that simplices are given by chains in $S$. The topological space $|K(S)|$ is called the geometrical realization of a poset $S$ and denoted $|S|$.
\end{defin}

\begin{rem}
The definitions of geometrical realizations of a poset and that of a simplicial complex agree. The complex $K(S_K)$ is the barycentric subdivision of $K$, therefore $|K(S_K)|\cong |K|$ for any simplicial complex $K$
\end{rem}

The following convention is quite common and natural: it is said that a simplicial complex $K$ or a poset $S$ ``has topological property $\ca{P}$'' if its geometrical realization $|K|$, or $|S|$ respectively, ``has topological property $\ca{P}$''. For example, $K$ is called contractible if $|K|$ is contractible.

\begin{rem}\label{remMinMaxCone}
Definition~\ref{defPosetRealiz} implies
\begin{equation}\label{eqCones}
|S_{\leqslant s}|=|s\ast S_{<s}|=\Cone|S_{<s}|,\mbox{ and similarly } |S_{\geqslant s}|=\Cone|S_{>s}|.
\end{equation}
In particular, this implies the following. If a poset $S$ has the largest element (i.e. $\hat{1}\in S$ such that $s\leqslant \hat{1}$ for any $s\in S$), then $S$ is contractible, $|S|\simeq\pt$. Indeed, if $\hat{1}$ is the largest element, then $|S|=|S_{\leqslant \hat{1}}|=\Cone|S_{<\hat{1}}|\simeq\pt$, according to \eqref{eqCones}. Similarly, the existence of the least element implies contractibility.
\end{rem}

\begin{rem}
Let $S$ be a poset, $K(S)$ be the corresponding simplicial complex, and $s\in S$. The following formula appears to be quite useful in practice:
\begin{equation}\label{eqLinkInPoset}
\lk_{K(S)}\{s\}=K(S_{<s})\ast K(S_{>s}).
\end{equation}
Indeed, a simplex of $\lk_{K(S)}\{s\}$ has the form $I=\{s_1,\ldots,s_r\}$ such that $I\sqcup \{s\}$ is a chain in~$S$. Therefore, we have (probably, after reordering) $s_1<\cdots<s_l<s<s_{l+1}<\cdots<s_r$. Therefore, $I$ is the concatenation of the simplex $\{s_1,\ldots,s_l\}$ of $K(S_{<s})$ and the simplex $\{s_{l+1},\ldots,s_r\}$ of $K(S_{>s})$.
\end{rem}

\subsection{Galois connection}

\begin{defin}\label{defMorphismPosets}
A function $F\colon S\to T$ between two posets is called a \emph{morphism} (or \emph{monotonic}), if it preserves the order: the condition $s_1\leqslant_S s_2$ implies $F(s_1)\leqslant_T F(s_2)$. A morphism is a functor from $\cat(S)$ to $\cat(T)$.
\end{defin}

If two morphisms $F\colon S\to T$ and $G\colon T\to S$ are given, they are simply written as $F\colon S\rightleftarrows T\colon G$.

\begin{defin}\label{defGaloisCon}
A pair of morphisms $F\colon S\rightleftarrows T\colon G$ is called a Galois connection, if\footnote{There are variations in the definition, slightly different from the one given here, e.g. by interchanging $F,G$ or changing the order relation to the opposite. Note that the functions $F$ and $G$ appear in the definition in non-symmetric way. One of them is usually called ``left'' and the other is ``right''. Usually it is completely impossible to memorize which one is the left and which is the right, hence this part of the definition is intentionally omitted.} for any $s\in S$ and $t\in T$ the condition $F(s)\geqslant_T t$ is equivalent to $s\geqslant_S G(t)$.
\end{defin}

When looking at $F,G$ as functors between $\cat(S)$ and $\cat(T)$, the Galois connection is nothing but the definition of an adjoint pair of functors.

\begin{rem}
An equivalent way of defining a Galois connection is to require that
\begin{equation}\label{eqGaloisEquiv}
[G(F(s))\leqslant_S s \mbox{ for any } s\in S]\mbox{ and }
[F(G(t))\geqslant_T t \mbox{ for any } t\in T].
\end{equation}
These conditions are often easier to check in practice.
\end{rem}

\section{Simplification using Galois connection}\label{secSimplification}

\subsection{Quillen--McCord theorem and the weeding operation}

We recall the classical result of Quillen~\cite{Quil} (which also appears in~\cite{McCord} in an equivalent form).

\begin{thm}[Quillen's theorem A, or Quillen's fiber theorem, or Quillen--McCord theorem]\label{thmQuilMcCord}
Let $F\colon S\to T$ be a morphism of finite partially ordered sets. Suppose that, for any $t\in T$, the geometrical realization  $|F^{-1}(T_{\geqslant t})|$ is contractible. Then $F$ induces the homotopy equivalence between $|S|$ and $|T|$.
\end{thm}

We refer to~\cite{Barmak} for a modern exposition of this result, and to~\cite{Bj} for related statements. A very short proof, based on the original technique of Quillen is given Section~\ref{secTopologicalFCA}.

\begin{cor}\label{corGalois}
Suppose that $F\colon S\rightleftarrows T\colon G$ is a Galois connection. Then both $F$ and $G$ induce homotopy equivalence between the spaces $|S|$ and $|T|$.
\end{cor}

\begin{proof}
By definition of the Galois connection we have
\[
F^{-1}(T_{\geqslant t})=\{s\in S\mid F(s)\geqslant t\} = \{s\in S\mid s\geqslant G(t)\}=S_{\geqslant G(t)}.
\]
The last poset has the minimal element $G(t)$, hence its geometrical realization is contractible by Remark~\ref{remMinMaxCone}. Then Quillen--McCord theorem applies. The proof for $G$ is completely similar (for this proof, Theorem~\ref{thmQuilMcCord} should be reformulated with $\leqslant$ instead $\geqslant$).
\end{proof}

Given a simplicial complex $K$, we can construct a (generally simpler) poset $\widetilde{K}$ which is Galois-connected with $K$ as follows.

\begin{con}\label{conMainConstruction}
Let $K$ be a simplicial complex. Let $W$ denote the set of simplices of $K$ which are maximal by inclusion (i.e. for any $I\in W$ there is no $J\in K$ such that $J\supset I$). Consider the subposet $\widetilde{K}\subset S_K=K\setminus\{\varnothing\}$ consisting of all simplices of $K$, which can be represented as the nonempty intersection of elements of $W$:
\[
\widetilde{K}=\left\{\bigcap\nolimits_{I\in A}I\neq\varnothing\mid A\subseteq W\right\}
\]
Consider two maps: the natural inclusion $F\colon \widetilde{K}\to K\setminus\{\varnothing\}$, and
\[
G\colon K\setminus\{\varnothing\}\to \widetilde{K},\qquad G(I)=\bigcap_{J\in W, J\supseteq I}J.
\]
It is easy to check the inequality \eqref{eqGaloisEquiv} for the constructed maps, so that the pair $(F,G)$ is a Galois connection. Therefore, $|K|\cong |S_K|\simeq|\widetilde{K}|$ by Corollary~\ref{corGalois}. We say that the poset $\widetilde{K}$ is obtained from $K$ by the \emph{weeding procedure}.
\end{con}

\begin{rem}
The above-mentioned construction may be efficient if we apply it to nerves of ``excess'' covers. Topological data analysis allows to investigate data clouds distributed on the unknown space $X$ by considering the nerves of a suitable contractible cover of $X$ and computing its simplicial homology. However, it is usually the case that the cover contains multiple intersections, so that its nerve has dimension, and the number of simplices, much bigger, than actually needed to compute its Betti numbers. This is where Construction~\ref{conMainConstruction} may find applications. We demonstrate this idea by a simple example.
\end{rem}

\begin{ex}\label{exCoverSimplify}
Let a 2-torus be covered by 9 disks as shown on Fig.~\ref{figNerveTorus}, left. This cover is contractible. Since there are quadruple intersections, the nerve $K$ of this cover has dimension $3$. The simplicial complex $K$ is shown on Fig.~\ref{figNerveTorus}, middle: it consists of $9$ tetrahedra, arranged periodically. This nerve has $9$ vertices, $36$ edges, $36$ triangles, and $9$ tetrahedra. However, to compute Betti numbers, it is natural to replace each tetrahedron on the picture by a square. This would result in a cell subdivision of a torus, having $9$ vertices, $18$ edges, and $9$ 2-dimensional cells, as shown on Fig.~\ref{figNerveTorus}, right. It is easy to show that the poset of cells of the right figure is exactly the poset $\widetilde{K}$ built from the simplicial complex $K$: each tetrahedron is treated combinatorially as a square, and all simplices which are not intersections of tetrahedra are neglected. Therefore, the passage from middle to right in Fig.~\ref{figNerveTorus} demonstrates Construction~\ref{conMainConstruction} of the weeding.
\end{ex}

\begin{figure}[h]
\begin{center}
\includegraphics[scale=0.3]{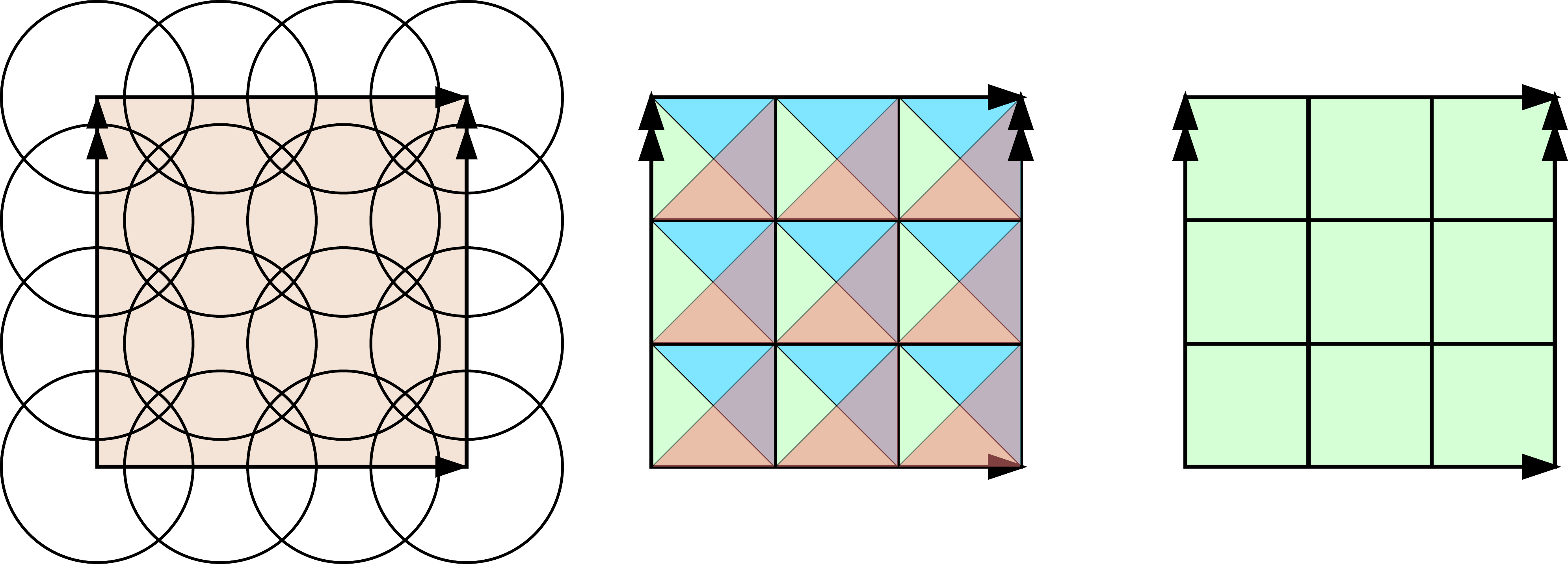}
\end{center}
\caption{A 2-torus covered by 9 discs and the nerve of the covering}\label{figNerveTorus}
\end{figure}

\begin{rem}
One remark should be made concerning the previous example. The poset $\widetilde{K}$, which we got in this example, is the poset of faces of some cell subdivision of the original space $X$. This allows to compute cellular homology of $X\simeq K\simeq \widetilde{K}$ from the chain complex of $\widetilde{K}$ (which has $\dim C_j(K)$ equal to the number of elements of $\widetilde{K}$ of rank $j$).

This may not be the case in general: it may happen that $\widetilde{K}$ is not a poset of faces of a cell subdivision. In this case, if we want to compute the homology of $\widetilde{K}$, we need to honestly pass to geometrical realization $|\widetilde{K}|$ and compute its simplicial homology. The $j$-dimensional simplices of the simplicial complex $|\widetilde{K}|$ are the chains in $\widetilde{K}$ having length $j+1$. The number of $j$-dimensional simplices of $|\widetilde{K}|$ may be bigger than the corresponding number for the original complex $K$: in this case the whole weeding algorithm does not make sense.
\end{rem}

\subsection{Simplification method: Stong reduction}

As the previous discussion shows, we can replace a simplicial complex by its subposet $\widetilde{K}$ without changing its homotopy type. Sometimes, the resulting poset can be further simplified by the iteration of the following construction. We attribute this construction to Stong, who introduced its analogue for finite topological spaces in~\cite{Stong}.

\begin{prop}\label{propStongReduction}
Assume that $s\in S$ has exactly one out-edge or exactly one in-edge in the Hasse diagram $\Gamma_S$ of the poset $S$. Then $|S\setminus s|\simeq |S|$.
\end{prop}

\begin{proof}
We tackle the case of out-edge, since the case of in-edge is completely similar (and follows from the first one by reversing the order). Let $e=(s,t)$ be the unique out-edge of $s$. Then $r>_Ss$ implies $r\geqslant_S t$. Therefore, $|S_{>s}|=|S_{\geqslant t}|$ is contractible according to Remark~\ref{remMinMaxCone}. We have
\[
|S|=|S\setminus s|\cup_{|\lk_{K(S)} s|}|\st_{K(S)} s|
\]
by~\eqref{eqStarSubcomplDecompose}. According to~\eqref{eqLinkInPoset}, there holds
\[
|\lk_{K(S)} s|=|S_{<s}|\ast|S_{>s}|.
\]
The space $|\lk_{K(S)} s|$ is contractible since $|S_{>s}|$ is contractible. Attachment of the contractible space $|\st_{K(S)} s|$ along the contractible subspace $|\lk_{K(S)}s|$ does not change the homotopy type, hence $|S|\simeq |S\setminus s|$.
\end{proof}

\begin{defin}\label{defUpDownCore}
If the element $s\in S$ satisfies the property that there exists $t>s$ such that $r>s$ implies $r\geqslant t$ (this means $s$ has unique out-edge in $\Gamma_S$), then $s$ is called an \emph{upbeat} in the terminology of \cite{May}, or \emph{linear} in the terminology of Stong~\cite{Stong}. Similarly, if there exists $t<s$ such that $r<s$ implies $r\leqslant t$ (this means $s$ has a unique in-edge), then $s$ is called a \emph{downbeat} or \emph{colinear}. If the poset $S$ does not have downbeats and upbeats it is called a \emph{core}.
\end{defin}

\begin{figure}[h]
\begin{center}
\includegraphics[scale=0.3]{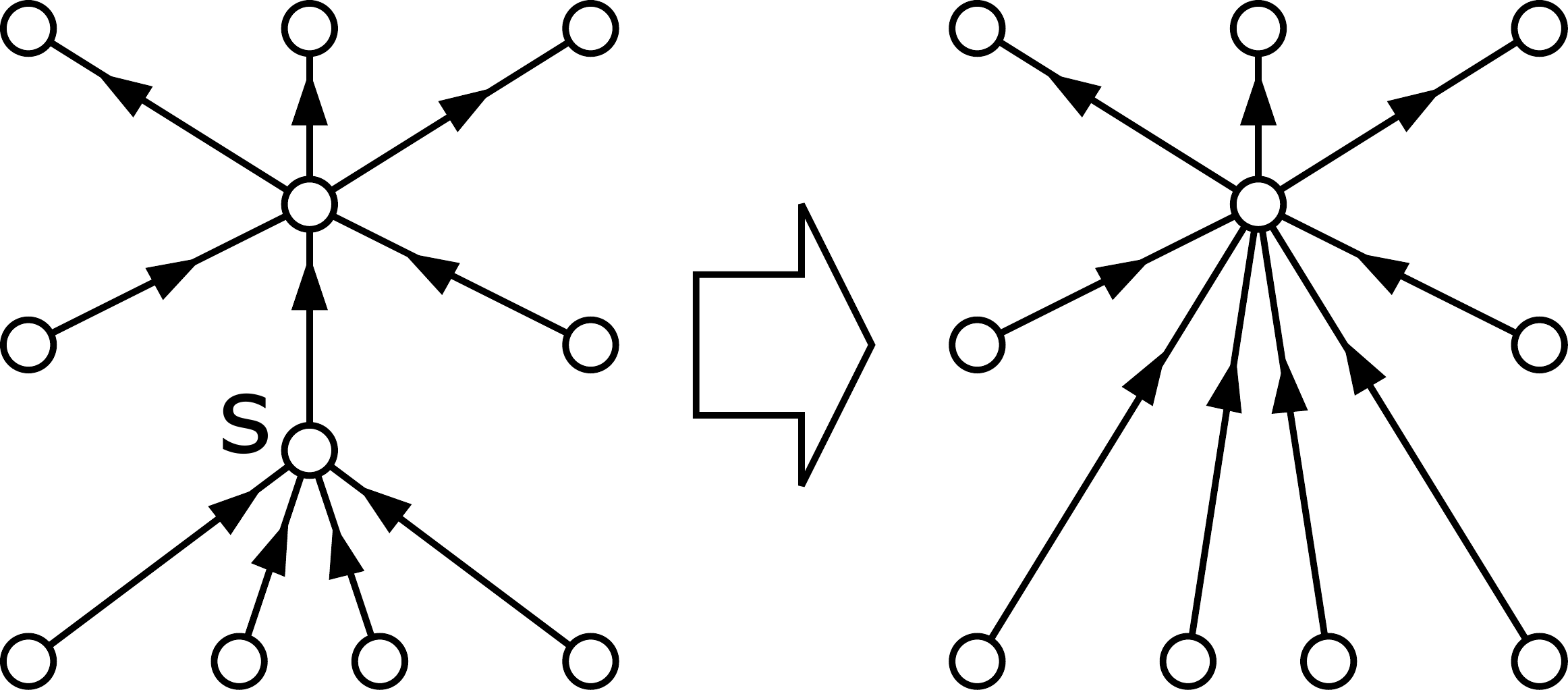}
\end{center}
\caption{Removing the upbeat $s$ from the Hasse diagram}\label{figUpbeatRemove}
\end{figure}

\begin{rem}
On the level of Hasse diagrams, removing the upbeat $s$ with the unique out-edge $e=(s,t)$ corresponds to contracting the edge $e$. All edges that entered $s$ will enter $t$ after this operation, see Fig.~\ref{figUpbeatRemove}. Sometimes this contraction produces redundant directed edges which should not be present in the diagram (since they can be represented by combinations of other arrows). Redundant arrows should be removed, see example on Fig.~\ref{figUpbeatRemoveAux}.
\end{rem}

\begin{figure}[h]
\begin{center}
\includegraphics[scale=0.3]{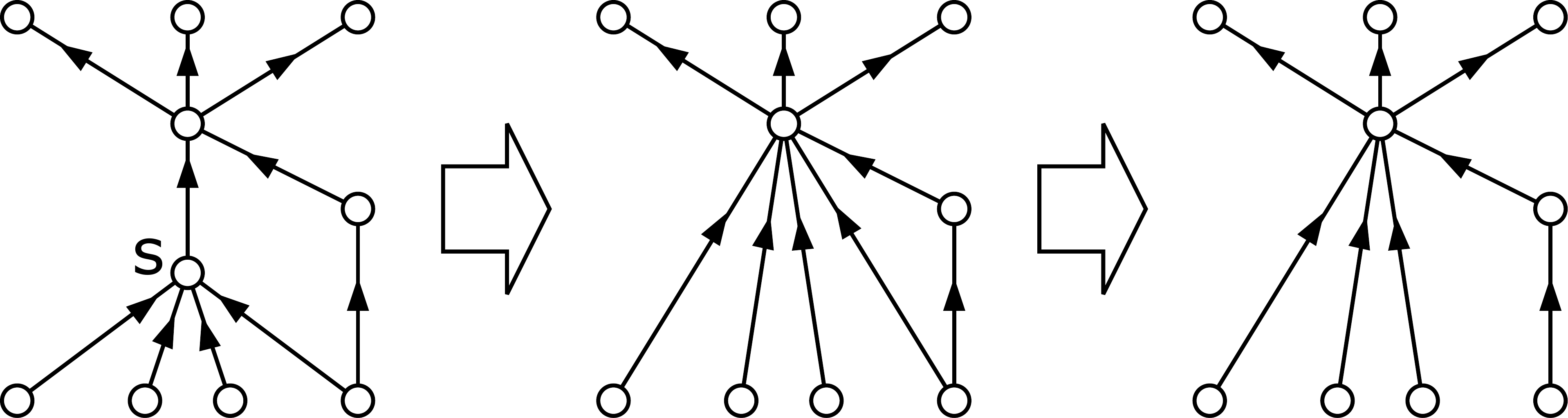}
\end{center}
\caption{Removing the upbeat $s$ and removing redundant arrows}\label{figUpbeatRemoveAux}
\end{figure}

\begin{defin}
If a poset $T$ is obtained from a poset $S$ by a sequence of downbeats' and upbeats' removals, we say that $T$ is obtained from $S$ by Stong reduction. If, moreover, $T$ is a core, then $T$ is called \emph{the core of} $S$ and denoted $\core S$.
\end{defin}

A core of $S$ is defined uniquely up to isomorphism, according to \cite[Thm.2 and Thm.3]{Stong}: it does not depend on the sequence of removals.

\begin{rem}
In the theory of Stanley--Reisner algebras of simplicial complexes there is a notion of a (simplicial) core of a simplicial complex, see~\cite[Def.2.2.15]{BPnew}. A simplicial complex $L$ is called the simplicial core of $K$, if, for some $s\geqslant 0$, there holds $K=L\ast\Delta_{[s]}$, and $L$ is not a cone. Here $\Delta_{[s]}$ is the full simplex on the vertex set $[s]$. It should be mentioned, that simplicial core does not coincide with the core given by Definition~\ref{defUpDownCore} in general. Indeed, if $K=L\ast\Delta_{[s]}$ and $s\geqslant 1$, then $K$ is contractible, while $L$ may not be contractible. So far, the simplicial core can change the homotopy type, while the core given by Definition~\ref{defUpDownCore} preserves it.
\end{rem}

We give another remark concerning Proposition~\ref{propStongReduction}.

\begin{rem}
The proof of Proposition~\ref{propStongReduction} suggests that there is a more general statement: if either $|S_{>s}|$ or $|S_{<s}|$ is contractible, then $s$ can be removed from $S$ without changing its homotopy type. In terms of finite topological spaces, this idea was proposed by Osaki~\cite{Osaki}. It should be noted however, that contractibility of a poset is difficult to confirm in practice.
However, there is a natural homological version which can be realized algorithmically. If either $|S_{>s}|$ or $|S_{<s}|$ is acyclic, then $s$ can be removed from $S$ without changing its homology type.
\end{rem}

%This construction works in the spirit of discrete Morse theory (see ), which was recently implemented for homology computations by Nanda

\subsection{Weeding is obtained by Stong reduction}
The weeding defined in Construction~\ref{conMainConstruction} can be obtained as a sequence of upbeat removals.

\begin{prop}\label{propWeedingStong}
The poset $\widetilde{K}$ is obtained from $S_K=K\setminus\{\varnothing\}$ by Stong reduction. It follows that $\core S_K$ is isomorphic to $\core\widetilde{K}$.
\end{prop}

\begin{proof}
For $I\in K$ consider ``the closure of $I$'' that is
\[
\cl(I)\eqd \bigcap_{J\in W,J\supseteq I}J\in \widetilde{K},
\]
where $W$ is the set of maximal simplices of $K$. In the notation of Construction~\ref{conMainConstruction}, this means $\cl=F\circ G\colon K\setminus\{\varnothing\}\to K\setminus\{\varnothing\}$. Clearly, $\cl(I)\supseteq I$. We show that all simplices $I$ such that $I\neq\cl(I)$ can be consecutively removed from $K$ by Stong reductions. Let us proceed by induction starting from the biggest possible cardinality $|I|$, so far, the induction hypothesis depends on the decreasing parameter $n$. Assume that all simplices $J\in K$ with $J\neq\cl(J)$ and $|J|>n$ have already been removed, and let us pick a simplex $I$ with $|I|=n$ and $I\neq\cl(I)$. For any $J>I$ we have $J=\cl(J)$ since otherwise $J$ would have been removed at the previous steps. Therefore $J>I$ implies $J=\cl(J)\geqslant \cl(I)$. This shows that $I$ is an upbeat, and it can be removed from the poset. Proceeding inductively, at the end we get the subposet $\widetilde{K}$ of $S_K$.
\end{proof}

The process in the proof of Proposition~\ref{propWeedingStong} is shown schematically on Fig.~\ref{figBoolRetraction}. If $I\neq\cl(I)$, then all the subsets $\{J\mid I\leqslant J<\cl(I)\}$ can be consecutively removed starting from the top layer.

\begin{figure}[h]
\begin{center}
\includegraphics[scale=0.25]{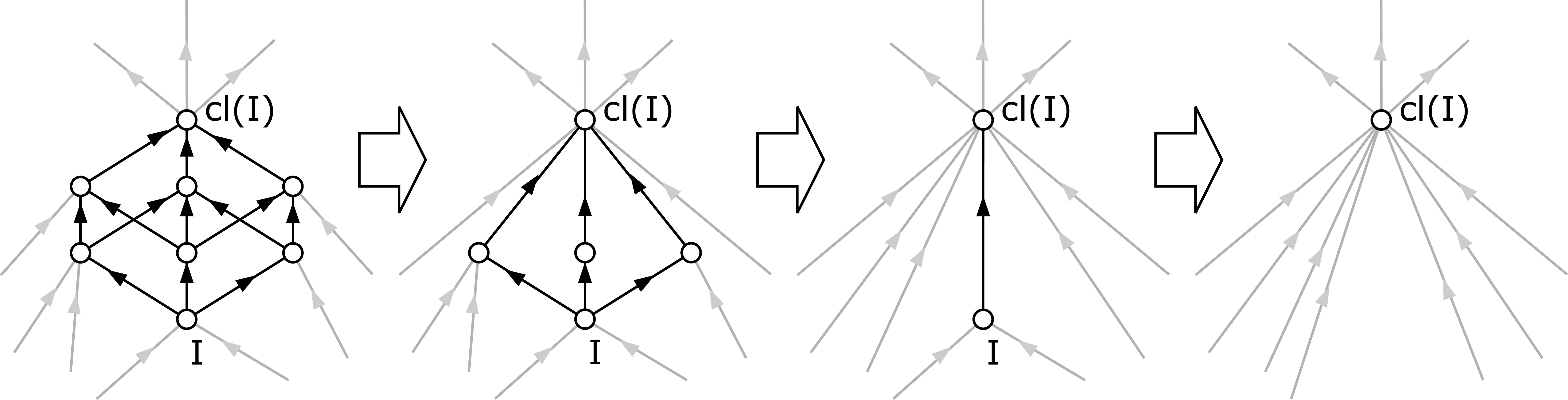}
\end{center}
\caption{Consecutive Stong retraction of the elements between $I$ and $\cl(I)$.}\label{figBoolRetraction}
\end{figure}

Propositions~\ref{propWeedingStong} and~\ref{propStongReduction} give the proof of the homotopy equivalence $|\widetilde{K}|\simeq |K|$ which is independent of the Quillen--McCord theorem.

\section{Nerve complexes of cellular spaces and their duality}\label{secNervesDuality}

\subsection{Nerve complexes of combinatorial objects}
First we give the definition of the nerve complex of a convex polytope, introduced in~\cite{AB}.

\begin{con}
Recall that a polytope is a convex hull of a finite set of points in some euclidean space $\Ro^n$, or, equivalently, a bounded intersection of a finite collection of closed affine half-spaces. Let $\ca{P}$ denote the poset of proper faces of $P$ (i.e. all faces except $\varnothing$ and $P$ itself). The convex polytope
\[
P^*=\{l\in(\Ro^n)^*\mid \langle l,x\rangle\geqslant -1 \mbox{ for all }x\in P\}
\]
is called \emph{polar dual} to $P$. $j$-dimensional faces of $P$ bijectively correspond to $(n-1-j)$-dimensional faces of $P^*$, so that the face poset $\ca{P}^*$ is isomorphic to $\ca{P}^{\op}$, the poset $\ca{P}$ with the reversed order.

An $n$-dimensional polytope is called simple if each of its vertices is contained in exactly $n$ facets (or, equivalently, in exactly $n$ edges). A polytope is called simplicial, if all its proper faces are simplices. A polytope dual to a simple polytope is simplicial and vice versa.
\end{con}

\begin{defin}
Let $P$ be a convex $n$-dimensional polytope and $\F_1,\ldots,\F_m$ be all its facets. The nerve $K_P$ of the covering $\bigcup_i\F_i$ of the boundary $\dd P$ is called \emph{the nerve complex of} $P$. In other words,
\[
K_P=\{I=\{i_1,\ldots,i_k\}\subset [m]\mid U_{i_1}\cap\ldots\cap U_{i_k}\neq\varnothing\}.
\]
\end{defin}

Since the nonempty intersections of facets of $P$ are the faces of $P$, which are contractible, Alexandrov Nerve theorem implies $K_P\simeq \dd P\cong S^{n-1}$. Moreover, if $P$ is simple, then $K_P$ coincides with the boundary of the dual polytope $P^*$. In this case, the nerve complex $K_P=\dd P^*$ is a simplicial sphere.

The nerve complexes were studied in \cite{AB} in connection with toric topology. In particular, nerve complexes allow to describe equivariant homotopy types of certain degenerate quadrics' intersections, called moment-angle spaces. Also, there was an attempt to use nerve complexes in the study of combinatorics of non-simple polytopes: this leads to certain formula, generalizing Dehn--Sommerville relations. As a byproduct, it was noticed that the nerve complex contains not only the homotopical information (which obviously follows from the Nerve theorem) but also the combinatorial information about the polytope. The combinatorics of $P$ can be completely reconstructed from the nerve complex $K_P$ by the following simple observation:
\[
\ca{P}^{\op}=\widetilde{K_P}.
\]
Recall that the poset $\widetilde{K_P}\subset K_P$ is the weeding of $K_P$ given by Construction~\ref{conMainConstruction}. In \cite{AB}, we proved the following

\begin{prop}\label{propNerveCpxPolytope}
If $I\in \widetilde{K_P}\subseteq K_P$ corresponds to a face $F\subset P$ of dimension $j$, then $\lk_{K_P}I$ is homotopy equivalent to the sphere $S^{j-1}$ (where by definition we put $S^{-1}=\varnothing$). Otherwise, i.e. for $I\in K_P\setminus(\widetilde{K_P}\cup\varnothing)$, the complex $\lk_{K_P}I$ is contractible.
\end{prop}

This gives an algorithmic way to determine, whether a simplex $I\in K_P$ belongs to $\widetilde{K_P}$ or not, without having to intersect all possible combinations of maximal simplices from $W$. The alternative
\[
\mbox{``a simplicial complex is contractible or homotopy equivalent to a sphere''}
\]
can be decided by computing its simplicial homology. Proposition~\ref{propNerveCpxPolytope} was further used in \cite{AyDepth} to show that the depth of the Stanley--Reisner algebra $\ko[K_P]$ (i.e. the maximal length of regular sequences in this algebra) coincides with $\dim P$.

Proposition~\ref{propNerveCpxPolytope} can be naturally extended to all simplicial complexes.

\begin{prop}\label{propLinksAcyclic}
Let $K$ be a simplicial complex. For every simplex $I\in K\setminus(\widetilde{K}\cup\varnothing)$, the link $\lk_KI$ is contractible.
\end{prop}

\begin{proof}
Consider the closure operator $\cl\colon K\to K$, defined by $\cl(I)=\bigcap_{J\in W,J\subseteq I}J$, as in the proof of Proposition~\ref{propWeedingStong}. By assumption, $I\notin \widetilde{K}$, therefore $\cl(I)$ is strictly larger than $I$. We have $\lk_KI=\Delta_{\cl(I)\setminus I}\ast \lk_K\cl(I)$, where $\Delta_{\cl(I)\setminus I}$ is the full simplex on the nonempty set $\cl(I)\setminus I$. Since $\cl(I)\setminus I$ is nonempty, $\lk_KI$ is a cone, hence contractible.
\end{proof}

The construction of nerve complexes can be extended from the class of boundaries of convex polytopes to a larger class of nice cell complexes. Let $Q$ be a finite cell complex.

\begin{defin}
A cell complex $Q$ is called \emph{combinatorial} if it satisfies two properties:
\begin{enumerate}
  \item $Q$ is regular, i.e. all attaching maps are injective.
  \item The intersection of any two closed cells of $Q$ is either empty or a unique closed cell of $Q$.
\end{enumerate}
\end{defin}

Let $S_Q$ denote the poset of closed cells of a combinatorial cell complex $Q$, ordered by inclusion. If $Q$ is combinatorial, then so are its $i$-skeleta, as well as the complexes $\dd F$ for any cell $F\subset Q$. Therefore, the standard induction argument shows that $Q\cong |S_Q|$.

\begin{defin}\label{defSimpClosure}
Let $Q$ be a combinatorial cell complex with the vertex set $V$. Then \emph{the simplicial closure} $\scl(Q)$ of $Q$ is defined as the simplicial complex on the set $V$, whose simplices have the form $\{i_1,\ldots,i_k\}\in \scl(Q)$ if and only if $i_1,\ldots,i_k$ lie in one closed cell of $Q$.
\end{defin}

This means that, to obtain $\scl(Q)$, one replaces each cell $F$ of $Q$ by a simplex $\scl(F)$ on the vertex set of $F$. For example, squares are replaced by tetrahedra, and so on.

\begin{prop}\label{propReplaceBySimplex}
If $Q$ is combinatorial, then $\scl(Q)$ is homotopy equivalent to $Q$.
\end{prop}

\begin{proof}
For any set of cells $F_1,\ldots,F_k$ of $Q$ there is an alternative: (1) there is no cell $F\subseteq Q$ such that all $F_i$ are its subcells; (2) there is a unique minimal cell $F=\bigvee F_i\subseteq Q$ containing all $F_i$. In the latter case, $\bigvee F_i$ can be obtained by intersecting all cells containing $F_i$: according to the definition, this gives a unique cell.

Now consider the covering of $|K(S_Q)|$ by the subsets $|\st_{K(S_Q)}\{v\}|$ for all vertices $v$ of $Q$ (recall that $K(S_Q)$ denotes the simplicial complex of chains of $S_Q$). We have either $|\st_{K(S_Q)}\{v_1\}|\cap\cdots\cap |\st_{K(S_Q)}\{v_k\}|=\varnothing$ or
\[
|\st_{K(S_Q)}\{v_1\}|\cap\cdots\cap |\st_{K(S_Q)}\{v_k\}|=\left|\st_{K(S_Q)}\scl\left(\bigvee\nolimits_{i=1}^kv_i\right)\right|,
\]
which is a cone, hence contractible. Therefore, the covering is contractible. It is easily seen that $\scl(Q)$ coincides with the nerve of the covering $\bigcup_v|\st_{K(S_Q)}\{v\}|=|S_Q|$, hence the Nerve theorem implies $\scl(Q)\simeq |S_Q|\cong Q$.
\end{proof}

In some cases Proposition~\ref{propReplaceBySimplex} can also be derived from the following proposition and Quillen--McCord theorem.

\begin{prop}\label{propReconCombCellCpx}
Let $Q$ be a combinatorial cell complex in which every cell is the intersection of some collection of maximal cells. Then the poset $S_Q$ is isomorphic to $\widetilde{\scl(Q)}$, the weeding of the simplicial complex $\scl(Q)$.
\end{prop}

The proof is straightforward from Construction~\ref{conMainConstruction} of the weeding and Definition~\ref{defSimpClosure}. The assumption in Proposition~\ref{propReconCombCellCpx} holds for manifolds as the following lemma shows.

\begin{lem}
If $Q$ is a combinatorial cell subdivision of a closed manifold, then every cell of $Q$ is the intersection of some collection of maximal cells.
\end{lem}

\begin{proof}
Consider the Poincare dual cell subdivision $Q^*$. Let $F^*$ be the face dual to $F$ and let $v_1,\ldots,v_s$ be all vertices of $F^*$. Then $F$ is the intersection of the maximal cells of $Q$ dual to $v_1,\ldots,v_s$.
%Assume the converse, i.e. there exists a cell $F$ of $Q$ which is not the intersection of maximal cells. Consider the cell $\cl(F)$ which is defined as the intersection of all maximal cells of $Q$ containing $F$. Then the face $\cl(F)$ strictly contains $F$.
%
%
%Let us consider all faces as the elements of $S_Q$. Take a face $\tilde{F}$ such that $F\leqslant \tilde{F}<\cl(F)$ and $\dim \tilde{F}=\dim\cl(F)-1$. We have $\cl(\tilde{F})=\cl(F)$ and the relation $G>\tilde{F}$ implies $$
\end{proof}

\begin{rem}\label{remBipartiteTables}
Propositions~\ref{propReplaceBySimplex} and~\ref{propReconCombCellCpx}, despite their simplicity, show an important fact: to reconstruct both homotopy type and combinatorics of certain combinatorial cell complexes, one only needs to know vertex lists of all maximal cells of $Q$. In this case, the simplicial complex $\scl(Q)$ can be reconstructed by adding all subsets of the maximal simplices, and the poset $S_Q$ can be reconstructed by considering all possible intersections of maximal simplices. Note that the list of maximal simplices can be encoded in a bipartite graph $(V\sqcup W,E)$ where $V$ is the set of vertices of $Q$, $W$ is the set of maximal cells of $Q$, and there is an edge $e=(v,w)$ between $v\in V$ and $w\in W$ if and only if $v\in w$. The graph can be encoded in a bit matrix $A$ of size $|V|\times|W|$, where $a_{vw}=1$ if $v\in w$ and $0$ otherwise. This way of representation of simplicial complexes gives certain benefits in persistent homology computations: see~\cite{BPP} and Remark~\ref{remSimplification} below.
\end{rem}

\subsection{Duality on general nerve complexes}

We start with a topologically motivated example generalizing the duality on polytopes.

\begin{ex}\label{exPoincareDual}
Let $Q$ be a combinatorial cell decomposition of a closed manifold $M$, $\dim M=n-1$, and $Q^*$ be the Poincare dual cell decomposition of the same manifold. Then the face poset $S_{Q^*}$ is isomorphic to $S_Q^{\op}$. The simplicial complex $\scl(Q)$ coincides with the nerve of the covering of $M$ by the maximal cells of $Q^*$. Similarly, $\scl(Q^*)$ is the nerve of the covering of $M$ by the maximal cells of $Q$, or, equivalently, the nerve of the covering of $M$ by the maximal simplices of $\scl(Q)$.

Let $L$ denote the nerve of the covering of $K=\scl(Q)$ by its maximal simplices. Then, according to Poincare duality, the poset $\widetilde{L}\cong S_{Q^*}$ is isomorphic to the poset $\widetilde{K}^{\op}\cong S_{Q}^{\op}$.
\end{ex}

The phenomenon described in this example is not just an instance of Poincare duality, but can be stated in much bigger generality. For completeness, we give the statements and proofs in combinatorial and topological manner, although, the ideas behind them are well developed in the area of formal concept analysis: the overview of this theory is given in Section~\ref{secFCA}.

\begin{thm}\label{thmDualityOfNerves}
Let $K$ be a simplicial complex and $L$ be the nerve of the covering of $K$ by its maximal simplices. Then the weeding poset $\widetilde{L}$ is isomorphic to the poset $(\widetilde{K})^{\op}$.
\end{thm}

\begin{proof}
Let $V$ and $W$ denote the vertex set and the set of maximal simplices of $K$ respectively. Therefore, $W$ is the vertex set of $L$. We denote the elements of $V$ by $i_1,i_2,\ldots$, and the elements of $W$ by $I_1,I_2,\ldots$. Consider the pair of morphisms $F\colon (2^V)^{\op}\rightleftarrows 2^W\colon G$, where
\begin{itemize}
  \item $2^X$ is the poset of all subsets of $X$ ordered by inclusion; $(2^X)^{\op}$ is the same set with the reversed order.
  \item $F(\{i_1,\ldots,i_k\})=\{I\in W\mid \{i_1,\ldots,i_k\}\subseteq I\}$;
  \item $G(\{I_1,\ldots,I_s\})=I_1\cap\cdots\cap I_s\subseteq V$ and $G(\varnothing)=V$.
\end{itemize}
It is easily checked that
\begin{equation}\label{eqGaloisProof1}
G(F(A))\leqslant A \mbox{ for any } A\in (2^V)^{\op} \mbox{ and}
\end{equation}
\begin{equation}\label{eqGaloisProof2}
F(G(B))\geqslant B \mbox{ for any } B\in 2^W
\end{equation}
which means that $F$ and $G$ form a Galois connection. Applying the morphism $F$ to inequality~\eqref{eqGaloisProof1} and substituting $B=F(A)$ into inequality~\eqref{eqGaloisProof2}, we get $F(A)\leqslant F(G(F(A)))\leqslant F(A)$, hence $F(G(F(A)))=F(A)$. Similarly, we have $G(F(G(B)))=G(B)$. This means that $F$ and $G$ are bijections between the images $F((2^V)^{\op})$ and $G(2^W)$.

The image $G(2^W)\subseteq (2^V)^{\op}$ coincides with $\widetilde{K}^{\op}\sqcup \{\varnothing, V\}$ according to the definition of the weeding.

Each vertex $i\in V$ determines the simplex $\sigma_i=\{I\subseteq W\mid I\ni i\}$ of $L$, and each maximal simplex of $L$ necessarily has the form $\sigma_i$ for some $i\in V$. It follows that elements of $\widetilde{L}$ have the form
\[
\sigma_{i_1}\cap\cdots\cap\sigma_{i_k}=\{I\subseteq W\mid \{i_1,\ldots,i_k\}\subseteq I\}=F(\{i_1,\ldots,i_k\}).
\]
Therefore, $F((2^V)^{\op})=\widetilde{L}\sqcup\{\varnothing, W\}$. Finally, we see that the maps $F$ and $G$ provide monotonic bijections between the posets $(\widetilde{K})^{\op}$ and $\widetilde{L}$.
\end{proof}

\begin{rem}\label{remSimplification}
Passing from a simplicial complex to its cover by maximal simplices is a natural thing to do if the complex is represented by a bipartite graph or a rectangular bit matrix, as in Remark~\ref{remBipartiteTables}. Roughly speaking, this operation interchanges two parts of a bipartite graph and transposes a matrix. There is, however, one detail missing. As was mentioned in the proof of Theorem~\ref{thmDualityOfNerves}, each vertex $i$ of $K$ determines the simplex $\sigma_i$ of $L$, and each maximal simplex of $L$ is determined this way. However, it may happen that $\sigma_{i_1}\subseteq \sigma_{i_2}$ for some $i_1\neq i_2$. Of course, the number of maximal simplices of $L$ in this case is strictly smaller than the number of vertices of $K$: one can remove the redundant simplex $\sigma_{i_1}$ from the list of maximal simplices.

This observation was used in~\cite{BPP} to obtain a simplification algorithm for the computation of persistent homology of a filtered simplicial complex. The algorithm (if we state it for one complex of the filtration) iterates the procedure of passing to the nerve of the covering by maximal simplices, and removes redundant simplices at each stage, until the procedure stabilizes. On a matrix level, this corresponds to transposing, and deleting a row if it is componentwise majorized by another row. Despite its seeming simplicity, this algorithm is claimed to speed up homology computations drastically.

It should be noticed however, that according to Theorem~\ref{thmDualityOfNerves}, the weeding poset $\widetilde{K}$ of $K$ does not change, up to order reversal, at each stage of the algorithm. Therefore, this poset gives an obstruction to the applicability of the algorithm. For example, if we take a combinatorial cell decomposition $Q$ of a closed manifold $M$, as in Example~\ref{exPoincareDual}, and consider the simplicial complex $\scl(Q)$, then the nerve of the maximal simplex cover of $\scl(Q)$ will be $\scl(Q^*)$. Iteration of this procedure, will successively switch between $\scl(Q)$ and $\scl(Q^*)$, hence no simplification occurs in this case.
\end{rem}

\section{Formal concept analysis and covers}\label{secFCA}
\subsection{Formal concepts}

Formal concept analysis (FCA) emerged in the work~\cite{Wil} and recently constitutes a vast research area in data mining. The results of FCA have found many applications in different areas. We refer to the book \cite{GaOb} for the modern exposition of the subject with the focus on algorithms. FCA methods are implemented in several computer programs, such as Concept Explorer (see the list of other tools in \cite[p.30]{GaOb}). This section aims to show how the notions of FCA are related to computational topology.

\begin{defin}\label{defFormalContext}
A \emph{formal context} is a triple $\Ca=(V,W,\I)$, where $V$ and $W$ are sets, and $\I\subseteq V\times W$ is a binary relation. The elements of $V$ are called objects, the elements of $W$ are called attributes, and we write $v\I w$ for the boolean expression $(v,w)\in \I$.
\end{defin}

Whenever $v\I w$ for $v\in V$, $w\in W$ (that is $(v,w)\in \I$), this is treated as ``the object $v$ has the attribute $w$''. Therefore, a formal context is nothing but a bipartite graph $\Gamma_\Ca$ on the vertex set $V\sqcup W$, or, equivalently, a bit matrix $M_\Ca$ of size $|V|\times|W|$. For a subset $A\subseteq V$ one can define a subset $A'\subseteq W$:
\[
A'=\{w\in W\mid\forall v\in A\colon v\I w\},
\]
the set of all common attributes of elements of $A$. Similarly, for any $B\subseteq W$, there is a subset $B'\subseteq V$:
\[
B'=\{v\in V\mid\forall w\in B\colon v\I w\},
\]
the set of all objects that having all attributes from the set $B$. This defines the pair of monotone (decreasing) morphisms
\[
F\colon V\rightleftarrows W\colon G, \quad F(A)=A',\quad G(B)=B',
\]
which is a Galois connection (Definition~\ref{defGaloisCon} should be restated for decreasing monotone functions by reversing the order in the left set).

\begin{defin}\label{definFormalConcept}
Let $A\subseteq V$, $B\subseteq W$. A pair $(A,B)$ is called a \emph{formal concept} if $A'=B$ and $B'=A$. The set $A$ is called the \emph{extent} of the formal concept $(A,B)$ and $B$ is called its \emph{intent}.
\end{defin}

One can think of a formal concept as a maximal by inclusion bi-clique (complete bipartite subgraph) in the graph $\Gamma_\Ca$, or as a maximal by inclusion rectangular submatrix of $M_\Gamma$ filled with $1$'s. In a concept, the extent and intent determine one another according to the definition.

\begin{con}
Formal concepts are partially ordered: $(A_1,B_1)\leqslant (A_2,B_2)$ if $A_1\subseteq A_2$ (or, equivalently, $B_1\supseteq B_2$). The poset of formal concepts in a given context $\Ca$ is denoted by $\mathfrak{B}(\Ca)=\mathfrak{B}(V,W,\I)$.

This has a natural logical and philosophical meaning: the smaller is the set of objects, the bigger is the set of their defining attributes. As an example, we consider the notion of ``a graph'' smaller than the notion of ``a cell complex''. On one hand, every graph is a cell complex: hence there are more cell complexes than graphs. On the other hand, a graph has more attributes than a cell complex: it has all attributes of a cell complex but, in addition, it is 1-dimensional.

In the theory of FCA, it is proved that the poset $\mathfrak{B}(\Ca)$ is a complete lattice, which means that any collection of concepts has the greatest common subconcept (the meet in a poset) and the least common superconcept (the join in a poset). Moreover, under certain conditions, a complete lattice can be represented as $\mathfrak{B}(\Ca)$ for some formal context (see The basic theorem of concept lattices, \cite{Wil}).
\end{con}

The discussion of the previous sections becomes more transparent in terms of FCA due to the following construction.

\begin{con}\label{conContextFromComplex}
Let $K$ be a simplicial complex. Consider the formal context $\Ca_K=(V,W,\I)$, where $V$ and $W$ are the vertex set and the set of maximal simplices of $K$ respectively, and $\I$ is the relation of inclusion: $(v,w)\in\I$ if and only if $v$ is a vertex of $w$. Then, by definition, the context lattice $\mathfrak{B}(\Ca_K)$ coincides with $\widetilde{K}\sqcup\{\hat{0},\hat{1}\}$. Here $\hat{0}=(\varnothing, W)$ is the least concept, and $\hat{1}=(V,\varnothing)$ is the greatest concept (recall that we did not consider them as the elements of the weeding $\tilde{K}$ according to Construction~\ref{conMainConstruction}).
\end{con}

\begin{rem}
On this way of thinking, the duality constructed in Theorem~\ref{thmDualityOfNerves} is a trivial statement. Indeed, there is a natural duality of formal contexts in which the roles of objects and attributes are interchanged. The formal concepts are defined in a way completely symmetric with respect to objects and attributes. The concept lattice remains the same under such interchange, only the order changes to its opposite. We emphasize the fact that duality between objects and attributes can be considered as a far going generalization of the Poincare duality.
\end{rem}

\begin{con}\label{conComplexFromContext}
It is possible to translate between FCA and topology in the opposite direction. If $\Ca=(V,W,\I)$ is a formal context, we can consider two simplicial complexes as follows. (1) $K_\Ca$ is the complex on the vertex set $V$ having simplices of the form $\{w\}'=\{v\in V\mid v\I w\}$ for $w\in W$ and all their subsets. In other words, $I\in K_\Ca$ if and only if $I'\neq\varnothing$. (2) $L_\Ca$ is the complex on the vertex set $W$ with simplices $\{w\in W\mid v\I w\}$ for $v\in V$ and their subsets.
\end{con}

Let $\Ca=(V,W,\I)$ be a formal context. Then there is a Galois connection
\begin{equation}\label{eqGaloisConcLatticeFull}
\cl\colon 2^V\rightleftarrows \mathfrak{B}(\Ca)\colon \inc,
\end{equation}
where $\cl(A)=(A'',A')$ and $\inc(A,B)=A$. Then Corollary~\ref{corGalois} of the Quillen--McCord theorem A asserts the homotopy equivalence $|2^V|\simeq |\mathfrak{B}(\Ca)|$. This statement, however, does not make much sense in general: both posets have the least element $\hat{0}$ (as well as the greatest element $\hat{1}$), hence their geometrical realizations are contractible by Remark~\ref{remMinMaxCone}. However, removing $\hat{0}$ and $\hat{1}$ from all posets, we get the following

\begin{prop}\label{propContextComplexesHomot}
Let $\Ca$ be a context such that both complexes $K_{\Ca}$ and $L_{\Ca}$ are not full simplices and do not have ghost vertices. Then we have the homotopy equivalences
\[
|K_\Ca|\simeq |\mathfrak{B}(\Ca)\setminus\{\hat{0},\hat{1}\}|=|\mathfrak{B}(\Ca)^{\op}\setminus\{\hat{0},\hat{1}\}|\simeq |L_\Ca|
\]
\end{prop}

\begin{proof}
The assumption tells that whenever the subset $A\subset V$ satisfies $A\neq\varnothing$ and $A\neq V$, we have $\cl(A)\neq \hat{0}=(\varnothing, W)$ and $\cl(A)\neq\hat{1}=(V,\varnothing)$, where the closure operator $\cl$ is defined in \eqref{eqGaloisConcLatticeFull}. Therefore we have a Galois connection
\begin{equation}\label{eqGaloisConcLatticeCut}
\cl\colon 2^V\setminus\{\varnothing,V\}\rightleftarrows \mathfrak{B}(\Ca)\setminus\{\hat{0},\hat{1}\}\colon \inc,
\end{equation}
proving the homotopy equivalence $|K_\Ca|\simeq |\mathfrak{B}(\Ca)\setminus\{\hat{0},\hat{1}\}|$. The second homotopy equivalence is proved similarly.
\end{proof}

\begin{rem}
The assumption of the Proposition~\ref{propContextComplexesHomot} means that the bit matrix of the formal context $\Ca$ does not have rows and columns consisting entirely of $0$'s or $1$'s.
\end{rem}

Proposition~\ref{propContextComplexesHomot} shows, that the whole subject of formal concept analysis is consistent with homotopy theory in some sense.

\begin{rem}
Note that Constructions~\ref{conContextFromComplex} and~\ref{conComplexFromContext} are not transverse to each other. When we pass from a formal context to a simplicial complex, in some sense, we forget the information about non-maximal simplices. Adding these simplices to the list does not affect simplicial complex, however, this may affect the concept lattice. In general, we have $K_{\Ca_K}=K$, but $\Ca_{K_{\Ca}}$ may not coincide with $\Ca$.
\end{rem}

%\begin{con}
%Let $\Ca$ be a formal context, and $\Gamma_\Ca$ --- the Hasse diagram of the concept lattice $\mathfrak{B}(\Ca)$.
%\end{con}

\subsection{Nerves and formal contexts}
We give an example of a formal context which naturally arises in topology.

%a natural definition and several remarks, which are well known to those familiar with homotopy theory, but may shed light on possible relations between FCA and topology. At first, we give an example of a formal context commonly used by topologists.

\begin{con}\label{conContextFromCovering}
Let $\ca{U}=\{U_i\}_{i\in[m]}$ be a covering of a topological space $X$, that is $X=\bigcup_{i\in[m]}U_i$. Consider the formal context $\Ca_{\ca{U}}=([m],X,\I)$, where $(i,x)\in \I$ if and only if $x\in U_i$.
\end{con}

\begin{rem}
Since $X$ is usually infinite, it is impossible to consider $L_{\Ca_{\ca{U}}}$ as a finite simplicial complex. However, it is possible to consider it as an infinite simplicial set as follows. With any set $M$, finite or infinite, one can associate ``the infinite dimensional simplex''
\[
\Delta_M=\colim\limits_{F\subseteq M, |F|<\infty} \Delta_F,
\]
with the topology of direct limit. Here $\Delta_F$ is the simplex on the vertex set $F$. It can be seen (e.g. by Whitehead theorem) that $\Delta_M$ is contractible. In the same manner, we can extend the definition of $L_\Ca$ for a context $\Ca=(V,W,\I)$ with infinite $W$ by setting
\begin{equation}\label{eqInfiniteNerve}
L_\Ca=\colim\limits_{\substack{F\subseteq W, |F|<\infty\\ F'\neq\varnothing}}\Delta_F.
\end{equation}
\end{rem}

It follows from the definition, that $K_{\Ca_{\ca{U}}}$ coincides with the nerve of the covering $\ca{U}$. The concept lattice $\mathfrak{B}(\Ca_{\ca{U}})$ contains the same homotopy information as $K_{\Ca_{\ca{U}}}$, however, it contains more combinatorial information as the following example shows.

\begin{ex}\label{exCoversLattices}
Consider three coverings shown on Fig.~\ref{figNerveConcepts}. In all three cases, the triple intersections are nonempty, therefore, the nerve of the covering is a triangle in all three cases. However, the concept lattices $\mathfrak{B}(\Ca_{\ca{U}})$ are different. We label each formal concept on the Hasse diagram by a its extent, i.e. by an element of $2^{[3]}$. The intents may be infinite sets that can be restored from extents. So we do not write intents in the labels.

For example, consider the figure in the middle. The element $13$ is not present on the Hasse diagram since
\[
\{1,3\}''=(U_1\cap U_3)'=\{i\in[3]\mid U_i\supseteq U_1\cap U_2\}=\{1,2,3\}\neq\{1,3\},
\]
and, therefore, $\{1,3\}$ is not an extent of any formal concept. Similarly, on the right figure, the set $U_2$ can be represented as the intersection of a larger collection of covering sets ($U_2=U_2\cap U_3$), hence $\{2\}$ does not belong to the concept lattice.
\end{ex}

\begin{figure}[h]
\begin{center}
\includegraphics[scale=0.3]{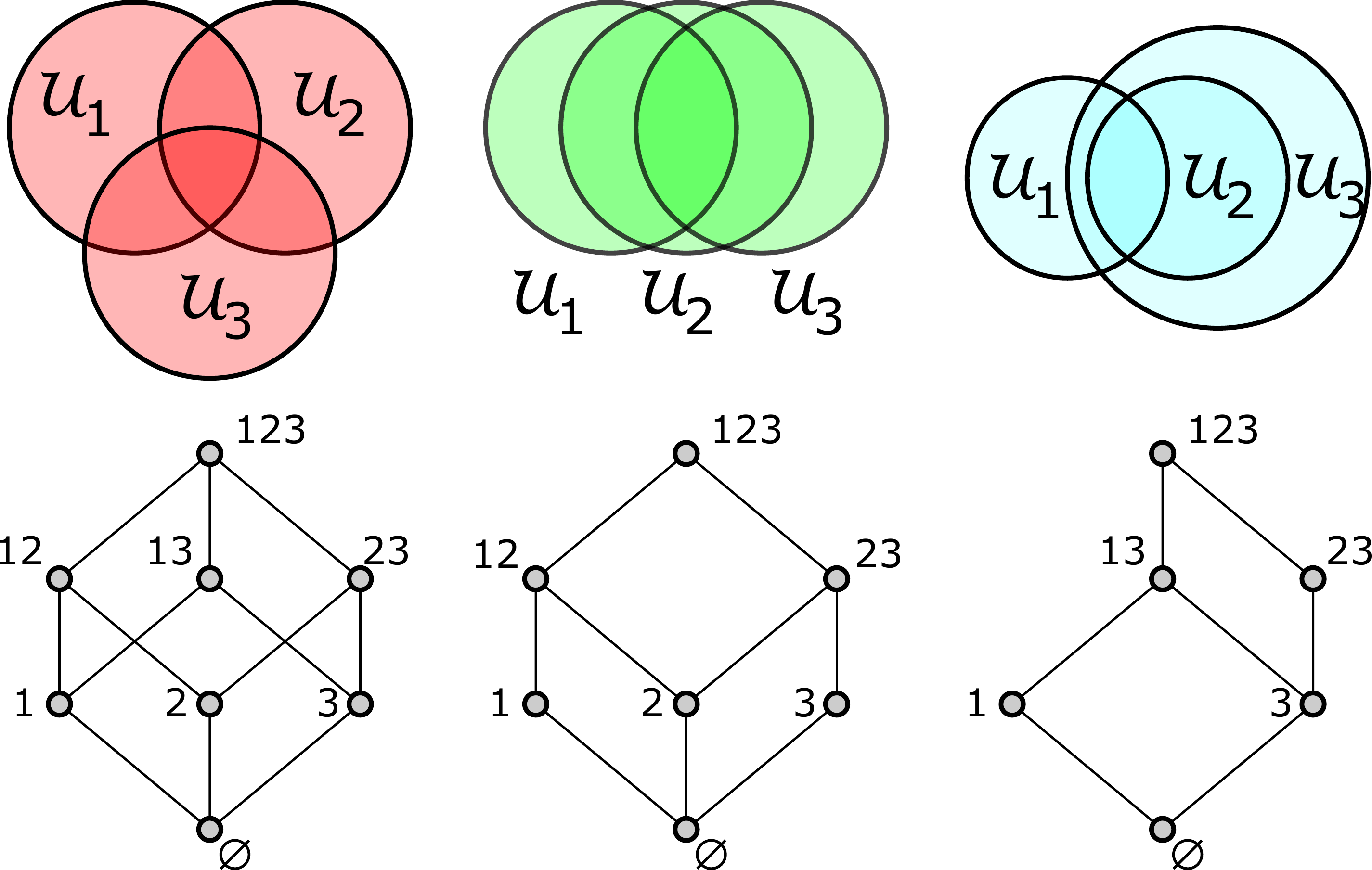}
\end{center}
\caption{Examples of coverings and their concept lattices}\label{figNerveConcepts}
\end{figure}

\section{Nerves, contexts, and brain studies}\label{secMouse}

The Nerve theorem and the related homotopy theory had found applications in brain study, more precisely in the study of neural activity of place cells. Originally, this topological approach in the study of place cells activity was proposed by Dabaghian, Cohn, and Frank~\cite{Dabag,DabagNoCov}; more recently this theory was further developed by Curto and Itskov~\cite{CurIts,CurItsNeuralRing}. There are other works, which are not related directly to place cells, however use similar techniques in similar tasks. We do not give a full list of relevant references, however, we mention the works of Ghrist~\cite{Ghr,Ghr2} on sensor networks (these are related to the problem of simplification of nerves of the coverings), the work of Remolina and Kuiper~\cite{RemKui} (this work describes the representation of the geometry of physical space in artificial networks), and the works~\cite{EAGN,EFP} (in which FCA in applied in the analysis of neural codes obtained by fMRI).

\begin{con}\label{conMouse}
The nerves of the coverings appear in brain study in the following way (we refer to Manin's expository work~\cite{Manin} for mathematical details and to~\cite{CurIts} for an exhaustive list of neurobiological references).

In the hyppocampus of a mammal, there exist special neural cells, called \emph{place cells}, first found by O'Keefe and Dostrovsky~\cite{Keefe}. Assume that some cell $i$ activates whenever the mammal is located in a specific region $U_i$ of the physical space $X$. In this case, $i$ is called a \emph{place cell}, and the subset $U_i\subset X$ is called its \emph{place field} (it should be mentioned, that more specific and precise definition is used by neurobiologists). In a typical experiment, a mouse is allowed to move freely inside the labyrinth of shape $X\subset\Ro^2$, while the activity of some part of its brain is recorded. Physically, place cells look the same as the other cells in the surrounding part of the hyppocampus: they can be distinguished only according to their functionality. Experiments show that there exist sufficiently many place cells, so the assumption that place fields $U_i$ cover $X$ is plausible.

The general task is to reconstruct the shape of the labyrinth $X$ from the data of neural activity, provided that the data are complete (it is assumed that the mouse had visited ``all points of the labyrinth'' sufficiently many times, and remembers the surrounding environment). We will assume that all place fields $U_1,\ldots,U_m$ are convex. This is not always the case, however, we may pick up only those $i$ for which $U_i$ is convex: there are still quite many of them. Then, theoretically, the task is easily solved by the Nerve theorem. Indeed, we initialize a simplicial complex $K$ on the vertex set $[m]$. Whenever we see that some set of neurons $\{i_1,\ldots,i_s\}$ is active simultneously, we conclude that the sets $U_{i_1},\ldots,U_{i_s}$ intersect, so we add the simplex $\{i_1,\ldots,i_s\}$ (and all its subsets) to $K$. We end up with the nerve $K$ of the covering $\{U_i\}$, which is homotopy equivalent to $X$.

So, in some sense, the physical space $X$ is encoded in the brain by means of the Nerve theorem. This discourse can be enlarged to a wider, though more theoretical extent as proposed by Manin~\cite{Manin}. Suppose that there exists a stimuli space $X$: some topological space, encoding all possible external events (stimuli) that may happen to an animal. Let us assume that a neuron $i$ activates when a stimulus from a subset $U_i\subset X$ acts on an animal (so that $i$-th neuron acts as an indicator function of $U_i$). Then, an animal has a nerve of the covering $\{U_i\}$ encoded in its brain. If we are lucky (for example, if the covering of $X$ by $U_i$'s is contractible), then the encoded picture is homotopy equivalent to the stimuli space $X$: we may conclude that an animal perceives reality in a homotopically correct way.
\end{con}

A related construction is also used in the analysis of place cells' activity patterns.

\begin{con}
In Construction~\ref{conMouse}, we only use the information on which sets of neurons activate simultaneously, however, we forget the information, which collections are simultaneously inactive at the same time. As was mentioned in Example~\ref{exCoversLattices}, the three coverings shown on Fig.~\ref{figNerveConcepts} all have the same nerves, however the combinatorics of coverings is different. In~\cite{CurItsNeuralRing} the neural codes were studied in a way which incorporates the information of inactivity. In general, given a covering $\{U_i\}_{i\in[m]}$ of $X$, we can define \emph{the neural code} as the subset of $2^{[m]}\cong\{0,1\}^m$ consisting of the bit strings
\[
\ca{N}_{\ca{U}}=\{(x\in U_1, x\in U_2,\ldots,x\in U_n)\in\{0,1\}^n\mid x\in X\}.
\]
In other words, $\ca{N}_{\ca{U}}=\{\{x\}'\mid x\in X\}$, if we use the language of the formal context $\Ca_{\ca{U}}$ defined in Construction~\ref{conContextFromCovering}.
\end{con}

The neural code contains a complete combinatorial information about the covering. In particular, the lower order ideal, generated by $\ca{N}_{\ca{U}}$ in the boolean lattice $2^{[m]}$, coincides with the nerve $K_{\ca{U}}$. The context lattice introduced in Construction~\ref{conContextFromCovering} can be recovered as well.

\begin{prop}\label{propCodeDeterminesConcepts}
The context lattice $\mathfrak{B}(\Ca_{\ca{U}})$ can be reconstructed from the neural code $\ca{N}_{\ca{U}}$ by the rule:
\[
(A,A')\in \mathfrak{B}(\Ca_{\ca{U}}) \Leftrightarrow \forall i\notin A \exists \tilde{A}\in \ca{N}_{\ca{U}} : A\subseteq\tilde{A} \mbox{ and } i\notin \tilde{A}.
\]
\end{prop}

The proof is straightforward.

\begin{rem}
In certain sense, the converse of Proposition~\ref{propCodeDeterminesConcepts} does not hold. The neural code cannot be extracted from the concept lattice, if the nodes of the lattice are labelled by their extents only. Fig.~\ref{figNerveCodesConcepts} demonstrates two coverings which have identical concept lattices but different neural codes. It is instructive to analyze the second picture. In this case, $U_3\subset U_1\cup U_2$ (in the terminology of FCA this means that the implication $U_3\to U_1\cup U_2$ holds in the context $\Ca_{\ca{U}}$). In other words, there is no point on the plane, where 3-rd neuron fires, while 1-st and 2-nd neurons are inactive. This is why we don't have $\{3\}$ in the neural code. However, $\{2\}$ appears in the concept lattice of this example. Indeed, there exists a couple of points (shown on Fig.~\ref{figNerveCodesConcepts}), which both lie in $U_3$, however this $2$-element set is neither a subset of $U_1$ nor $U_2$.
\end{rem}

\begin{figure}[h]
\begin{center}
\includegraphics[scale=0.3]{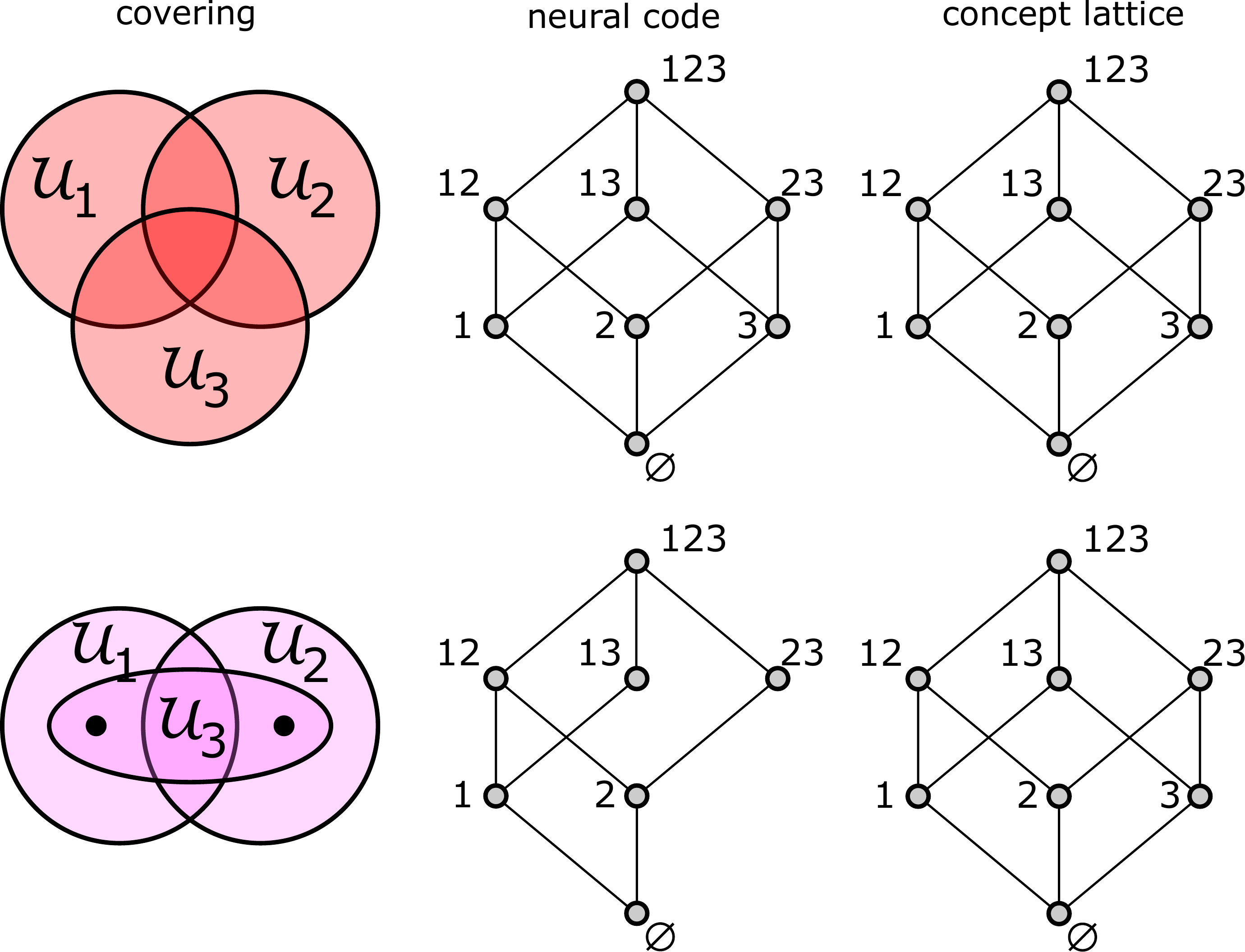}
\end{center}
\caption{Two coverings, their neural codes, and concept lattices}\label{figNerveCodesConcepts}
\end{figure}

\begin{rem}
The paper~\cite{CurItsNeuralRing} studies neural codes from the viewpoint of commutative algebra and algebraic geometry. Note that any neural code $\ca{N}_{\ca{U}}\subseteq 2^{[m]}$ can be considered as a subset of the affine space $\Fo_2^m$ over the finite field with two elements. Hence $\ca{N}_{\ca{U}}$ is an algebraic variety and its coordinate ring $\Rn=\Fo_2[\ca{N}_{\ca{U}}]=\Fo_2[v_1,\ldots,v_m]/\In$ is defined. The ring $\Rn$ is called \emph{the neural ring} of $\ca{N}_{\ca{U}}$, and $\In$ is the \emph{neural ideal}. Certain boolean statements about the covering, such as whether $\bigcap_{i\in I}U_i$ is a subset of $\bigcup_{j\in J}U_j$ for the given index sets $I,J\subseteq[m]$, can be answered in terms of whether a certain polynomial $m_{I,J}$ belongs to the neural ideal $\In$ or not. Hypothetically, this allows to apply Gr\"{o}bner bases for data mining in neural codes.

We make a remark that whether an inclusion $\bigcap_{i\in I}U_i\subseteq U_j$ holds for a cover $\ca{U}$, can be read from the concept lattice $\mathfrak{B}(\Ca_{\ca{U}})$, even if the intents are not specified in the labelling of nodes. In the terminology of FCA, the condition $\bigcap_{i\in I}U_i\subseteq U_j$ is formulated as
\[
\mbox{The implication }\bigcap\nolimits_{i\in I}U_i \to U_j \mbox{ holds in the context }\Ca_{\ca{U}}.
\]
It is not difficult to observe that
\begin{equation}\label{eqInclusion}
\bigcap\nolimits_{i\in I}U_i\subseteq U_j\mbox{ if and only if } j\in I''
\end{equation}
(see~\cite[Sec.3.1]{GaOb}). Given $I\subset[m]$, we can find the closure $I''$, by selecting all elements of the lattice $\mathfrak{B}(\Ca_{\ca{U}})$ with extents containing $I$; then $I''$ is their meet in the lattice:
\[
I''\mbox{ is the extent of }\bigwedge_{(J,J')\in \mathfrak{B}(\Ca_{\ca{U}}), J\supseteq I}J.
\]
We suggest that finding relations between neural rings' approach and data retrieval algorithms developed in FCA may be a promising direction of research.
\end{rem}

\section{Topological formal contexts}\label{secTopologicalFCA}

\subsection{Contexts with structure}

Sometimes there is a necessity to consider formal contexts in which both the set of objects and the set of attributes have some internal structure: a topology, a partial order, an action of a fixed group, or something else.

\begin{rem}
In Example~\ref{exCoversLattices}, the space of attributes $X$ is a topological space, which makes the intents of all concepts (i.e. intersections $U_{i_1}\cap\cdots\cap U_{i_k}$ topological spaces as well).

In the neurobiological application given by Construction~\ref{conMouse}, we have the formal context $(M,X,\I)$, where $M=[m]$ is the (finite) set of neurons and $X$ is the physical space. We see that $X$ carries the natural topological structure, which should be taken into account if we want to use homotopy machinery for the analysis of neural codes. In Construction~\ref{conMouse}, the set $M$ did not have any structure. This is unnatural from biological viewpoint: physical neurons are located in the brain, and they are organized in some structures by themselves. Speaking theoretically, we can take the whole brain, or its parts, as the set $M$. In this case there is a topology on $M$ as well. This may either be a ``euclidean'' topology (in which two neurons are close if they are located physically close to each other) or some sort of ``connectivity'' topology (in which two neurons are close, if the signal passes from one to another in a short time) or any other reasonable topology. In any case, it is natural to assume that both the set of objects $M$ and the set of attributes $X$ have topological structure.

This setting may seem too theoretical at first glance. Looking at the brain as topological continuum, we can hardly imagine an algorithm to deal with it: we may only study discrete objects. However, it should be noted, that working with finite topological spaces instead of continuous ones is not a big restriction from the homotopy-theoretical point of view. The basis of homotopy theory of finite topological spaces was laid in the works of McCord~\cite{McCord} and Stong~\cite{Stong}. McCord~\cite{McCord} had shown that weak homotopy types of finite CW-complexes are efficiently stored in finite (although non-Hausdorff) topological spaces. For example, there is a 4-point topological space, which is weakly homotopy equivalent to the circle. Stong~\cite{Stong} had shown that a reasonable notion of homotopy equivalence exists for a class of finite topological spaces (so called $T_0$-spaces, or Alexandrov topologies~\cite{Alex}).

There is an increased interest to the homotopy theory of finite topologies nowadays: we refer to the books of Barmak~\cite{BarmakBook} and May~\cite{May} and references therein. In particular, we mention that Quillen--McCord theorem takes a natural form, being restated in terms of finite topologies, as was done by McCord~\cite{McCord}. To summarize: the setting in which both objects and attributes have topology seems meaningful from the algorithmic viewpoint as well.
\end{rem}

\begin{defin}\label{defTopolFormalContext}
A \emph{topological formal context} is a triple $\Ca=(X,Y,\I)$, where $X$ and $Y$ are topological spaces, and $\I\subseteq X\times Y$ is a topological subspace.
\end{defin}

For certain tasks, the subset $\I$ is assumed to have some additional properties which assure that it is tame enough (for example as desribed in subsection~\ref{subsecProofs}). Definition~\ref{defTopolFormalContext} does not differ too much from Definition~\ref{definFormalConcept}: we just assume that all sets have topology. The notion of a formal concept also remains the same, however, in this case extents and intents of formal concepts are topological spaces.

Similarly, we can suppose that objects and attributes have partial order.

\begin{defin}
An \emph{ordered formal context} is a triple $\Ca=(X,Y,\I)$, where $X$ and $Y$ are posets, and $\I\subseteq X\times Y$ is a subposet with the induced order.
\end{defin}

The concepts are defined as in Definition~\ref{definFormalConcept}. Extents and intents of all concepts inherit partial order.

\begin{rem}
It is possible to give a general pointless definition of an $\ST{M}$-enriched formal context $\Ca=(X,Y,\I)$, where $\ST{M}$ is a well-powered complete symmetric monoidal category with an initial object. In this general setting, $X$ and $Y$ are objects of $\ST{M}$ and $\I$ is a subobject in the product $X\otimes Y$ (probably, taken from a certain class of subobjects). The concepts are given by Definition~\ref{definFormalConcept} where we put $A'$ to be the class of the monomorphism
\[
f_A=\colim\limits_{B\in\SubObj(Y)\colon A\times B\subseteq \I}B\to Y,
\]
(we should require that $f_A$ is a monomorphism so that its class determines a subobject). The subobject $B'\in\SubObj(X)$ is defined similarly for $B\in\SubObj(Y)$.

On this way, we can get a variety of formal concept theories, such as topological ($\ST{M}=\ST{Top}$, the category of topological spaces), ordered ($\ST{M}=\ST{Poset}$, the category of posets), quantum ($\ST{M}=\Hilb$, the category of Hilbert spaces), etc.
\end{rem}

\subsection{Appendix: Classical proofs using topological contexts}\label{subsecProofs}

In this final subsection we review the classical proofs of Alexandrov Nerve theorem and Quillen--McCord theorem in terms of topological contexts. First we recall the principal result of Smale~\cite{Smale}. A topological space $X$ is called locally contractible if, for any point $x\in X$ and any neighborhood $U\ni x$ there exists a smaller neighborhood $V\subset U$ of $x$ such that the inclusion $V\hookrightarrow U$ is homotopic to a constant map.

\begin{thm}[Smale~\cite{Smale}]\label{thmSmale}
Let $f\colon A\to B$ be a proper surjective map of locally compact separable metric spaces and assume that $A$ is locally contractible. If the fiber $f^{-1}(b)\subseteq A$ is contractible and locally contractible for all $b\in B$, then $f$ is a homotopy equivalence.
\end{thm}

If $X$ is a finite CW-complex, then $X$ is an ANR, which implies local contractibility of $X$ according to~\cite{Kod}. Therefore we have the following.

\begin{cor}\label{corContractFibers}
Let $f\colon A\to B$ be a cellular map of finite CW-complexes and, for each $b\in B$, the fiber $f^{-1}(b)$ is a contractible CW-complex. Then $f$ is a homotopy equivalence.
\end{cor}

Nerve theorem and Quillen--McCord theorem are proved using the same idea. The argument is classical (see \cite[Cor.4G.3]{Hat}, \cite[Theorem A]{Quil}), although we restate it in terms of topological formal contexts. To prove the homotopy equivalence between two spaces $X$ and $Y$, we find a topological formal context $(X,Y,\I)$, $\I\subseteq X\times Y$, such that the projections $\pr_X\colon \I\to X$ and $\pr_Y\colon \I\to Y$ are homotopy equivalences. In view of Corollary~\ref{corContractFibers}, if all fibers
\[
\{x\}'=\pr_X^{-1}(x)\quad\mbox{and}\quad \{y\}'=\pr_Y^{-1}(y)
\]
are contractible CW-complexes for all $x\in X$ and $y\in Y$, then $X\simeq \I\simeq Y$.

\begin{proof}[Proof of the Nerve theorem (Theorem~\ref{thmNerveThm})]
Let $\{U_i\}_{i\in[m]}$ be the covering of the space $X$ and $K_{\ca{U}}$ be its nerve. We assume that $X$ is a CW complex and all intersections $\bigcap_{i\in I}U_i$, $I\in K_{\ca{U}}$ are its CW-subcomplexes (although the proof is similar, if $X$ is arbitrary and the subsets $U_i$ are open). If $x$ is a point in $K_{\ca{U}}$, let $I(x)\subseteq [m]$ be the unique simplex of the nerve, which contains $x$ in its interior.

Consider the topological formal context $(K_{\ca{U}},X,\I_{\Nerve})$, where
\[
\I_{\Nerve}=\left\{(x,y)\in K_{\ca{U}}\times X\mid y\in \bigcap\nolimits_{i\in I(x)}U_i\right\}.
\]
Then, for any $x\in K_{\ca{U}}$, the space $\{x\}'=\pr_{K_{\ca{U}}}^{-1}(x)=\bigcap_{i\in I(x)}U_i$ is contractible by assumption. For any $y\in X$, the space $\{y\}'=\pr_Y^{-1}(y)$ coincides with the simplex on the vertex set $\{i\in[m]\mid U_i\supset x\}$. Hence $\{y\}'$ is contractible as well. Therefore, applying the Corollary~\ref{corContractFibers} twice we get $X\simeq I_{\Nerve}\simeq K_{\ca{U}}$.
\end{proof}

\begin{rem}
Without the assumption that the cover is contractible, there still holds $X\simeq I_{\Nerve}$, by the same argument. For general covers, the space $I_{\Nerve}$ is the homotopy colimit of the diagram $D\colon \cat(S_K^{\op})\to\ST{Top}$, $D(\{i_1,\ldots,i_s\})= U_{i_1}\cap\cdots\cap U_{i_s}$.
\end{rem}

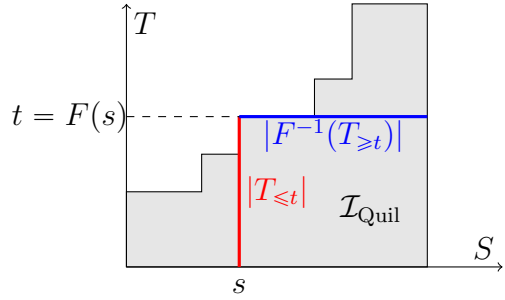
\begin{figure}[h]
\begin{center}
\begin{tikzpicture}[scale=0.5]
        \filldraw[fill=black!10] (0,0)--(0,2)--(2,2)--(2,3)--(3,3)--(3,4)--(5,4)--(5,5)--(6,5)--(6,7)--(8,7)--(8,0)--cycle;
        \draw[->]  (0,0)--(10,0);
        \draw[->]  (0,0)--(0,7);
        \draw (9.5,0.5) node{$S$};
        \draw (0.5,6.5) node{$T$};
        \draw[very thick, red] (3,0)--(3,4);
        \draw[very thick, blue] (3,4)--(8,4);
        \draw[dashed] (0,4)--(3,4);
        \draw (3,-0.5) node{$s$};
        \draw (-1.5,4) node{$t=F(s)$};
        \draw[red] (4,2) node{$|T_{\leqslant t}|$};
        \draw[blue] (5.5,3.5) node{$|F^{-1}(T_{\geqslant t})|$};
        \draw (6.5,1.5) node{$\I_{\Quil}$};
%
%
%        \draw (2.5,0.5) node{$\ast$}; \draw (3.5,0.5)
%        node{$\ast$}; \draw (4.5,0.5) node{$\ast$}; \draw (4.5,-0.5)
%        node{$\ast$}; \draw (4.5,-1.4) node{$\vdots$}; \draw (4.5,-2.5) node{$\ast$};
%
%        \draw[->] (4.2,0.5)--(3.7,0.5); \draw[->]
%        (4.3,-0.4)--(2.7,0.3); \draw[->] (4.3,-2.3)--(0.7,0.3);
%
%        \draw (-0.8,1.5) node{$\Ea{Q}^1_{*,*}$}; \draw (0.5,1.25) node{\tiny
%        $0$}; \draw (2.5,1.25) node{\tiny $n-2$}; \draw (3.5,1.25) node{\tiny
%        $n-1$}; \draw (4.5,1.25) node{\tiny $n$};
%
%        \draw (-0.2,0.5) node{\tiny $0$}; \draw (-0.2,-0.5) node{\tiny
%        $-1$}; \draw (-0.55,-2.5) node{\tiny $-(n-1)$};
%
%        \draw (4.3,0.2) node{$\difa{Q}^1$}; \draw (3.5,-0.5)
%        node{$\difa{Q}^2$}; \draw (2.5,-1.5) node{$\difa{Q}^n$};
\end{tikzpicture}
\end{center}
\caption{The proof of Quillen--McCord theorem} \label{figQuillenThm}
\end{figure}

\begin{proof}[Proof of Quillen--McCord theorem (Theorem~\ref{thmQuilMcCord})]
We have a morphism of posets $F\colon S\to T$ such that, for any $t\in T$, the space $|F^{-1}(T_{\geqslant t})|$ is contractible. We will prove that $|S|\simeq |T|$ (this is slightly weaker than proving $f$ to be a homotopy equivalence). Consider the topological formal context $(|S|,|T|,\I_{\Quil})$ where
\[
\I_{\Quil}=\bigsqcup_{\substack{I=(s_1<\cdots<s_k)\in K(S)\\J=(t_1<\cdots<t_l)\in K(T)\\t_l\leqslant f(s_1)}}|\sigma|\times|\tau|\subseteq |S|\times |T|.
\]
The context $\I_{\Quil}$ can be understood as the undergraph of the morphism $f$ (Fig.~\ref{figQuillenThm} gives a rough idea of this construction). For a point $x$ of the geometrical realization of a poset $R$, consider the unique simplex $\sigma=(s_0<s_1<\cdots<s_k)$ which contains $x$ in its interior, and define $m(x)\eqd s_0$ and $M(x)\eqd s_k$. Then, for any $x\in |S|$, the space $\{x\}'=\pr_{|S|}^{-1}(x)=|T_{\leqslant F(m(x))}|$ is contractible by Remark~\ref{remMinMaxCone}. For any $y\in |T|$, the space $\{y\}'=\pr_{|T|}^{-1}(y)=|F^{-1}(T_{\geqslant M(y)})|$ is contractible by assumption. Corollary~\ref{corContractFibers} implies homotopy equivalences $|S|\simeq \I_{\Quil}\simeq |T|$.
\end{proof}

\section*{Acknowledgements}
I am grateful to Konstantin Anokhin for bringing the subject of topological analysis of neural codes to my attention, and to Sergei Ivanov for telling me about the papers~\cite{Dabag,DabagNoCov}. I am also grateful to Oleg Kachan for the introduction to the study of strong collapses in computational topology.

\end{document}